\documentclass{article}

\usepackage{epic}
\usepackage{eepic}
\usepackage{mathtools}
\usepackage{ mathrsfs }
\usepackage{amsmath}
\usepackage{amsthm}
\usepackage{amssymb}
\usepackage{enumerate}
\usepackage{tikz-cd}

\usepackage{xcolor}

\usepackage{hyperref}
\usepackage{url}

\newcommand{\al}{\alpha}

\newcommand{\C}{\mathbb C}
\newcommand{\Q}{\mathbb Q}
\newcommand{\R}{\mathbb R}
\newcommand{\Z}{\mathbb Z}

\newcommand{\A}{\mathbb A}

\newcommand{\Ima}{\operatorname{Im}}
\newcommand{\Rea}{\operatorname{Re}}

\newcommand{\Aut}{\operatorname{Aut}}
\newcommand{\Hom}{\operatorname{Hom}}

\newcommand{\tr}{\operatorname{tr}}
\newcommand{\rk}{\operatorname{rk}}

\newcommand{\GL}{\operatorname{GL}}
\newcommand{\PGL}{\operatorname{PGL}}
\newcommand{\q}{\operatorname{q}}

\newcommand{\Stab}{\operatorname{Stab}}
\newcommand{\Std}{\operatorname{Std}}

\DeclareMathOperator{\spin}{spin}
\DeclareMathOperator{\id}{id}
\DeclareMathOperator{\Sat}{Sat}
\DeclareMathOperator{\Rep}{Rep}
\DeclareMathOperator{\Spin}{Spin}

\DeclareMathOperator{\diag}{diag}
\newcommand{\define}{=}
\newcommand{\enifed}{=}
\newcommand{\Ort}{\operatorname{O}}

\newcommand{\Sort}{\operatorname{SO}}

\newcommand{\Obs}{\operatorname{Obs}}

\newcommand{\matr}[4]{\left(\begin{smallmatrix}
		#1 & #2 \\
		#3 & #4
	\end{smallmatrix}\right)}
\newcommand{\SGamma}{\mathrm{S}\Gamma}
\newcommand{\SDelta}{\mathrm{S}\Delta}
\newcommand{\sumstack}[1]{\sum_{\substack{#1}}}

\newcommand{\SL}{\operatorname{SL}}
\newcommand{\Sp}{\operatorname{Sp}}

\newcommand{\fin}{\mathrm{f}}
\newcommand{\mfg}{\mathfrak{g}}
\newcommand{\mfk}{\mathfrak{k}}
\newcommand{\mfp}{\mathfrak{p}}
\newcommand{\autrep}{\Pi_{\mathrm{cusp}}}
\newcommand{\holautform}[2]{\mathcal{A}_{k,\mathrm{hol}}(#1,#2)}

\newcommand{\comment}[1]{}

\numberwithin{equation}{section}

\newtheorem{thm}{Theorem}[section]
\newtheorem{defthm}[thm]{Definition and Theorem}
\newtheorem*{con*}{Conjecture}
\newtheorem{prp}[thm]{Proposition}
\newtheorem{cor}[thm]{Corollary}
\newtheorem{lem}[thm]{Lemma}

\theoremstyle{definition}

\title{Eisenstein series for orthogonal groups and Kneser neighbours}
\author{Manuel K.-H. M\"uller and Nils R. Scheithauer}

\begin{document}

\begin{center}
	{\Large\bf Eisenstein series for orthogonal groups\\[2mm] and Kneser neighbours}\\[10mm]
	Manuel K.-H.\ M\"{u}ller,\\
	Dipartimento di Matematica "Guido Castelnuovo"\\
	Universit{\'a} di Roma "La  Sapienza", Rome, Italy \\[10mm]
	Nils R.\ Scheithauer,\\
	Fachbereich Mathematik,\\
	Technische Universit\"{a}t Darmstadt, Darmstadt, Germany
\end{center}

\vspace*{1.2cm}
\begin{center}
	\begin{minipage}{.8\textwidth}
		\noindent
		{\small We study Eisenstein series for orthogonal groups of even unimodular lattices of signature $(n,2)$ which arise as lifts of elliptic modular forms for $\SL_2(\Z)$ on 1-dimensional cusps. When the weight is small they have to be defined by analytic continuation. In general, they will be smooth but non-holomorphic after regularization. We show exactly in which cases we obtain holomorphic modular forms by computing the intertwining factors appearing in Langlands' constant term formula in terms of standard parameters and Rankin--Selberg $L$-values. This is related to Kneser's neighbour theory of positive-definite even lattices, which also governs the eigenvalues of these Eisenstein series. In particular, we describe the space of singular weight modular forms and construct the function $\Phi_{12}$, the denominator function of the Fake Monster Algebra, as an Eisenstein series. We apply these results to show that the corresponding orthogonal modular varieties are of general type.}
	\end{minipage}
\end{center}

\vspace*{1.2cm}
\begin{tabular}{rl}
	1 & Introduction \\
	2 & Preliminaries and definitions \\
	3 & Hecke operators and constant terms \\
	4 & Eisenstein series \\
	5 & Holomorphic modular forms \\
	6 & Geometric applications 
\end{tabular}
\vspace*{10mm}

\section{Introduction}

Orthogonal modular forms are holomorphic functions on the Hermitian symmetric space $\Ort_{n,2}(\R)/(\Ort_n(\R)\times\Ort_2(\R))$ which are invariant up to a cocycle under a discrete subgroup $\Gamma\subset \Ort_{n,2}(\R)$. Hence, they are automorphic forms for the reductive group $\Ort_{n,2}$ with special regularity properties. Geometrically they arise naturally as global sections of tensor powers of the Hodge line bundle of the orthogonal Shimura variety $\Gamma\backslash\Ort_{n,2}(\R)/(\Ort_n(\R)\times\Ort_2(\R))$. This makes them interesting from the point of view of arithmetic geometry since many moduli spaces have natural models as orthogonal Shimura varieties, such as for example the moduli space of quasi-polarized K3-surfaces \cite{Huybrechts}. On the other hand, they also appear in the representation theory of generalized Kac--Moody algebras (see e.g.\ \cite{BorcherdsMM}, \cite{GritsenkoNikulin1}, \cite{GritsenkoNikulin2}, \cite{ScheithauerInventiones}). Borcherds found two powerful methods to construct orthogonal modular forms, the additive and the multiplicative theta lift (see \cite[Theorems 13.3, 14.3]{Borcherds98}).

We consider the case that $\Gamma$ is the orthogonal (or special orthogonal)
group of the even unimodular lattice $I\!I_{n,2}$ of signature $(n,2)$, $n>2$. Then the variety $\Gamma\backslash\Ort_{n,2}(\R)/(\Ort_n(\R)\times\Ort_2(\R))$ is non-compact and can be compactified by adding finitely many $0$- and $1$-dimensional cusps. The $1$-dimensional cusps are parameterized by the genus $I\!I_{n-2,0}$ of positive-definite even unimodular lattices of rank $n-2$ and can be identified with modular curves $\SL_2(\Z)\backslash\mathbb{H}$. Accordingly, restricting orthogonal modular forms to these cusps yields modular forms for $\SL_2(\Z)$. In this article we study those orthogonal modular forms which do not vanish on the $1$-dimensional cusps, but vanish on the unique $0$-dimensional cusp. As we will show, their Hecke eigenvalues encode information about Kneser neighbour numbers of the genus $I\!I_{n-2,0}$. These modular forms are somewhat illusive since they are not in the image of the additive theta lift. We construct them as Eisenstein series.

Let $\widetilde{I\!I}_{n-2,0}$ denote the set of isometry classes of positive-definite even unimodular lattices with respect to $\Sort_{n-2}$ and let $\C[\widetilde{I\!I}_{n-2,0}]$ be the space of formal linear combinations of classes in $\widetilde{I\!I}_{n-2,0}$. The two eigenspaces of the outer automorphism coming from $\Sort_{n-2}\backslash\Ort_{n-2}$ are denoted by $\C[I\!I_{n-2,0}^{1}]$ and $\C[I\!I_{n-2,0}^{\det}]$. The restriction to the cusps is a linear map $\Phi$ on the space of modular forms of weight $k$ taking values in $\mathrm{M}_k(\SL_2(\Z))\otimes_\C\C[\widetilde{I\!I}_{n-2,0}]$. For an isometry class $[D]\in I\!I_{n-2,0}$ we denote by $\Gamma_D\subset\Gamma$ the (parabolic) subgroup fixing the cusp associated to $[D]$. The group $\Gamma$ has two characters, the trivial character $1$ and $\det$. Let $\chi$ be one of them and pick $f\in\mathrm{S}_k(\SL_2(\Z))$. We define the Eisenstein series $E_{f,[D],\chi}(Z,s)$ as
\begin{multline*}
	E_{f,[D],\chi}(Z,s) \\
	\define \sum_{\gamma\in\Gamma_D\backslash\Gamma}\chi(\gamma)f((-\rho_D,\cdot))|_k[\gamma](Z)\!\left(\frac{-\q(\Ima(Z))}{|j(\gamma,Z)|^2(-\rho_D,\Ima(\gamma Z))}\right)^{s}
\end{multline*}
and extend linearly to $\mathrm{S}_k(\SL_2(\Z))\otimes_\C\C[\widetilde{I\!I}_{n-2,0}]$.
This series converges absolutely for all $\Rea(s) > n-1-k/2$ to a smooth function in $Z$ and has a meromorphic continuation to all $s\in\C$. For $f\in\mathrm{S}_k(\SL_2(\Z))\otimes_\C\C[\widetilde{I\!I}_{n-2,0}]$ we denote this function by $E_f(Z,s)$. If $k>2n-2$, then $E_f(Z)\define E_f(Z,0)$ is a holomorphic modular form of weight $k$ with $\Phi(E_f) = f$. We investigate what happens for small $k$ at the point $s=0$.

The regularity of $E_f(Z,s)$ at $s=0$ depends on the inducing data. We may restrict to the case $f\otimes v\in\mathrm{S}_k(\SL_2(\Z))\otimes_\C\C[\widetilde{I\!I}_{n-2,0}]$ where $f\in\mathrm{S}_k(\SL_2(\Z))$ is a simultaneous eigenform and $v\in\C[\widetilde{I\!I}_{n-2,0}]$ is an eigenvector for the Hecke algebra of $\Sort_{n-2}$ acting on $\C[\widetilde{I\!I}_{n-2,0}]$. Their Hecke eigenvalues are Kneser neighbour numbers. According to the work of Arthur on the classification of automorphic representations by endoscopic transfer to $\GL_N$ \cite{Arthur}, the automorphic representation attached to $v$ is encoded by a standard parameter $\psi$ of the form
\[ \psi = \bigoplus_{i}\pi_i[d_i] \]
where $\pi_i\in\Pi_{\mathrm{cusp}}(\GL_{n_i})$ are self-dual cuspidal automorphic representations, $d_i\in\Z_{\geq1}$ and $\sum_{i}n_id_i = n-2$ (see p.\ \pageref{sec:StandardParameters}). The first main result of this article is the following (see Theorem \ref{thm:EisensteinseriesClassicalIsAdelic}).

\medskip

\noindent\textbf{Theorem.}
{\em Let $f\in\mathrm S_k(\SL_2(\Z))$ be a normalized eigenform and $v\in\C[\widetilde{I\!I}_{n-2,0}]$ a non-zero Hecke eigenvector. Assume that $v$ has Arthur parameter $\psi=\bigoplus_{i=1}^m\pi_i[d_i]$. Then $E_{f\otimes v}(Z)$ is a holomorphic modular form if and only if one of the following conditions hold
	\begin{enumerate}[(i)]
		\item $k \geq n+2$
		\item $k=n$ and $L(\frac{1}{2},\pi_f\otimes\pi_i)=0$ for some $i$ with $d_i=1$
		\item $k = n-2$ and $L(\frac{1}{2},\pi_f\otimes\pi_i)\not=0$ whenever $d_i = 1$
		\item $k < n-2$, $\pi_i\cong\pi_f$ and $d_i=n-k-2$ for some $i$ and $L(\frac{1}{2},\pi_f\otimes\pi_i)\not=0$ whenever $d_i = n-k-1$.
	\end{enumerate}
	In this case 
	\[ \Phi(E_{f\otimes v}) = f\otimes v. \]}

\medskip

Note how this contrasts to the case of Siegel modular forms, where the holomorphicity of Eisenstein series of Klingen type depends only on the weight $k$ but not on the specific inducing data (see \cite{Weissauer}).

The multiplicity formula of Arthur describes the discrete spectrum of a reductive group in terms of standard parameters. In particular, it describes for which Arthur parameters there exists a vector $v\in\C[\widetilde{I\!I}_{n-2,0}]$ with this parameter. Note that the multiplicity formula is still conditional on the hypothesis
\begin{enumerate}[(H1)]
	\item The weighted fundamental lemma (WFL) is true for non-split groups and its non-standard version is true in general.
\end{enumerate}
We will state whenever we assume (H1). See p.\ \pageref{sec:MultiplicityFormulaDependencies} for more details.

We show that any orthogonal modular form whose value at the cusps is a cusp form and which is orthogonal to the kernel of $\Phi$ is a linear combination of Eisenstein series (see Proposition \ref{prp:EisensteinBasis}). Assuming (H1) they are classified by the above theorem (Theorem \ref{thm:EisensteinseriesClassicalIsAdelic}).

Of particular interest are modular forms of singular weight $n/2-1$. This is the smallest possible weight for which there can exist non-trivial holomorphic modular forms. A famous example is
$\Phi_{12} = \Psi_{I\!I_{26,2}}^{\mathrm{mult}}(1/\Delta)$, the multiplicative theta lift of $1/\Delta$. This modular form has weight 12 and is the unique automorphic product of singular weight on a unimodular lattice \cite{ScheithauerCompositio}. It is the denominator function of the Fake Monster Algebra, a generalized Kac--Moody algebra describing the physical states of a bosonic string moving on the torus $\R^{(25,1)}/I\!I_{25,1}$. We can now give a complete description of the space of singular-weight modular forms on unimodular lattices (see Theorem \ref{thm:SingularWeightEisensteinSeries}).

\medskip

\noindent\textbf{Theorem.}
{\em Suppose that (H1) is true. Let $f\in\mathrm{S}_{n/2-1}(\SL_2(\Z))$ be a normalized eigenform. There exists an element $v_f^1\in\C[I\!I_{n-2,0}^1]$ and $v_f^{\det}\in\C[I\!I_{n-2,0}^{\det}]$ with Arthur parameter $\psi = \pi_f[n/2-1]$. The Eisenstein series $E_{f\otimes v_f^\chi}(Z)$ has singular weight and the space of singular weight modular forms on $I\!I_{n,2}$ is given by
	\[ \C\Psi_{I\!I_{n,2}}^{\mathrm{add}}(1) \oplus\bigoplus_{\substack{f\in\mathrm{S}_{n/2-1}(\SL_2(\Z)) \\ \text{normalized eigenform}}}\left(\C E_{f\otimes v_f^1} \oplus \C E_{f\otimes v_f^{\det}}\right). \]
	If $L(\frac{1}{2},\pi_f,\Std_{\GL_2}) \not=0$, then $v_f^\chi\in\C[I\!I_{n-2,0}^\chi]$ exists without the assumption of (H1) and $\vartheta_{I\!I_{n-2,0}}^{(n/2-1)}(v_f^1)$ is the Ikeda-lift of $f$.}

\medskip

In particular, we have $\Phi_{12} = E_{\Delta\otimes v_\Delta^{\det}}$, which gives a direct proof of the modularity of the sum expansion of $\Phi_{12}$. We also remark that for $n=26$, the description of the singular weight forms is entirely unconditional on (H1) (see Corollary \ref{cor:SingularWeightFormsn=26}).

Using Arthur parameters, Chenevier and Lannes showed in \cite[Corollary 5.3.2]{ChenevierLannes} that the difference of the number of times the Leech lattice is a proper $p$-neighbour of itself and the number of times it is an improper $p$-neighbour of itself is given by
\[  \#\mathcal{N}_{\Z/p\Z}^+([\Lambda],[\Lambda]) - \#\mathcal{N}_{\Z/p\Z}^-([\Lambda],[\Lambda]) = \tau(p)\frac{p^{12}-1}{p-1}
\]
where $\tau(p)$ denotes the $p$-th Fourier coefficient of the $\Delta$-function. We generalize this result to arbitrary signatures in Theorem \ref{thm:NeighbourFormulaUsingOrthogonalModularForms}.

\medskip

\noindent\textbf{Theorem.}
{\em Let $F\in\mathrm{M}_{n/2-1}(\Gamma,\chi)$. Suppose that $\Phi(F) = f\otimes v$ for some modular form $f\in\mathrm{M}_{n/2-1}(\SL_2(\Z))$ with $c_f(1) = 1$ and $v\in\C[I\!I_{n-2,0}^\chi]$. Then
	\[ T_{\Z/p\Z}v = c_f(p)\frac{p^{n/2-1}-1}{p-1} \,  v. \]}

\medskip 

Taking $F=\Phi_{12}$ recovers Chenevier and Lannes' result. The theorem imposes strong restrictions on the existence of singular-weight forms.

The Kodaira dimension of a variety is a birational invariant measuring the size of its canonical model. Gritsenko, Hulek and Sankaran \cite{GritsenkoHulekSankaran3} showed that the Kodaira dimension of $\Gamma_n\backslash\mathcal{H}$, where $\Gamma_n = \Ort(I\!I_{n,2})$, is $\kappa(\Gamma_n\backslash\mathcal{H}) = -\infty$ if $n=2,10,18$ and $\kappa(\Gamma_n\backslash\mathcal{H}) = n$ if $n=42, 50, 58, \ldots$ The only missing cases are $n=26$ and $34$. Using the Eisenstein series of singular weight we can construct cusp forms of sufficiently high vanishing order to obtain our fourth main result (see Theorems \ref{thm:GeneralType26} and \ref{thm:GeneralType34}).

\medskip

\noindent\textbf{Theorem.}
{\em The Kodaira dimension of $\Gamma_{26}\backslash\mathcal{H}$ is
	\[   \kappa(\Gamma_{26}\backslash\mathcal{H}) = 26,  \]
	i.e. $\Gamma_{26}\backslash\mathcal{H}$ is of general type. If the number of chiral lattices in $I\!I_{32,0}$ is bounded by $10^{33}$, then 
	\[   \kappa(\Gamma_{34}\backslash\mathcal{H}) = 34.  \] }

\medskip

The paper is organized as follows. The main computations regarding Eisenstein series are done in the adelic setting in terms of automorphic forms on adele groups. In Section \ref{sec:Definitions} we introduce the necessary notation and definitions for orthogonal groups and automorphic forms. In Section \ref{sec:HeckeOperatorsAndConstantTerms} we recall some theory of Hecke algebras and constant terms of automorphic forms. The results on Eisenstein series are obtained in Section \ref{sec:EisensteinSeries} by computing their constant terms applying Langlands' constant term formula. In Section \ref{sec:HolomorphicModularForms}
we study holomorphic modular forms and derive the first three main results described above. In the final section we apply our results to prove the general type results for the orthogonal Shimura varieties. 

\medskip

The authors would like to thank M.\ Dittmann, P.\ Kiefer, J.\ H.\ Bruinier and G.\ Chenevier for stimulating discussions and helpful comments.

Both authors acknowledge support by the Deutsche Forschungsgemeinschaft (DFG, German Research Foundation). The first through a \emph{Walter Benjamin-programme fellowship}, project number 566117541, the second through the CRC \emph{Geometry and Arithmetic of Uniformized Structures}, project number 444845124.

In the creation of this article, ChatGPT was used for literature research and discussions. All mathematical ideas and results were obtained by the authors.

\section{Preliminaries and definitions}\label{sec:Definitions}

We recall some definitions and results and set the notation for this work. For backgrounds we refer to \cite{GetzHahn}.

\medskip

Let $L$ be an even lattice of signature $(n,2)$ with $n>2$ and rational ambient space $V = L\otimes_\Z\Q$. We assume that $V$ has Witt rang $2$. Denote the bilinear form on $L$ and $V$ by $(\cdot,\cdot)$ and set $\q(v) = (v,v)/2$ for $v\in V$. We define an affine group scheme over $\Z$ by
\[ G(R) = \Sort(L\otimes_\Z R) \]
for any commutative ring $R$ with unit. It is an affine reductive group over $\Q$. Let $F$ and $F'$ be two isotropic subspaces of $V$ of dimension $2$ that are dual to each other. The space $W = ( F \oplus F')^{\bot}$ is a positive-definite subspace of $V$ of dimension $n-2$. We set $\Sort_W(R) \define\Sort(D_0\otimes_\Z R)$ where $D_0 = L\cap W$. We choose bases $e_1,e_2$ of $F$ and $f_1,f_2$ of $F'$ such that
\[ (e_1,e_2) = (f_1,f_2) = 0,\qquad (e_i,f_j) = \delta_{ij}  \]
and define $Z = \Q(e_1-f_1) \oplus \Q(e_2-f_2)$. The setwise stabilizer $P = \Stab_{G}(F)$ is a maximal parabolic subgroup of $G$. It has a Levi decomposition $P = MN$ where
\[  M = \GL_F \times \, {\Sort}_W   .  \]
We write the elements of $M$ as $(m,\omega) \in \GL_2\times \Sort_W$ where the action of $(m,\omega)$ on $V = \langle e_1,e_2\rangle \oplus W \oplus \langle f_1,f_2\rangle$ is given by $(m,\omega,(m^{-1})^T)$. On the other hand $N(R)$ is given by
\[ \{n(x_1,x_2,t) \define E^{e_1}_{te_2+x_1}E^{e_2}_{x_2}\mid x_1,x_2\in W(R), t\in R \} \]
where $E^{e_i}_{x_i}\in G(R)$ are called \emph{Eichler-transformations} and are given by
\[ E^{e_i}_{x_i}(v) \define v - (v,e_i)x_i + (v,x_i)e_i - \q(x_i)(v,e_i)e_i \]
for $x_i\in V(R)\cap\langle e_i,f_i\rangle^\bot$.

The group $G$ has a maximal $\Q$-split torus unique up to conjugation which we choose as  
\[ T = \{ (\diag(a_1,a_2),1)\in M \mid a_i \in \mathbb{G}_m \} \subset M. \]
We write $(\diag(a_1,a_2),1) = t(a_1,a_2)$ and define characters $\varepsilon_i:T \to \mathbb{G}_m$, $t(a_1,a_2) \mapsto a_i$ so that 
\[  X^*(T) = \Z \varepsilon_1 \oplus \Z \varepsilon_2. \]
We can decompose the Lie algebra $\mathfrak{g}$ of $G$ under the adjoint action of $T$ into eigenspaces
\[   \mathfrak{g} = \mathfrak{m} \oplus \bigoplus_{\alpha \in \Phi(G,T)} \mathfrak{g}_{\alpha}. \]
with weights $\Phi(G,T) \subset X^*(T)$ given by   
\[
\pm(\varepsilon_1-\varepsilon_2), \qquad
\pm(\varepsilon_1+\varepsilon_2), \qquad
\pm\varepsilon_1, \qquad
\pm\varepsilon_2 .
\]
They have multiplicities $n-2$ resp.\ $1$ depending on whether the weight is long or short. The weights form a root system of type $B_2$ in  $X^*(T) \otimes_{\Z} \R$. In the following we refer to them as roots. We choose
\[  \al_1 = \varepsilon_1-\varepsilon_2, \qquad \al_2 = \varepsilon_2  \]
as simple roots. Then the positive roots are $\al_1, \al_2, \al_1 + \al_2, \al_1 + 2 \al_2$. The Weyl group $W(\Phi(G,T))$ is the dihedral group of order $8$. It is isomorphic to the \emph{relative Weyl group} $W(G,T)\define N_G(T)(\Q)/Z_G(T)(\Q)$. It acts on $X^*(T)$ by signed permutations of $\varepsilon_1,\varepsilon_2$. Furthermore, we set
\[ W(M,M)\define N_{W(G,T)}(W(M,T))/W(M,T)=\{1,w\} \]
where $w$ as an element of $W(\Phi(G,T))$ acts as
\[ \varepsilon_1\mapsto-\varepsilon_2,\qquad\varepsilon_2\mapsto-\varepsilon_1. \]
Equivalently, it is the group of Weyl isomorphisms of the split center of $M$ induced by elements carrying $M$ to itself. Let $\rho_P$ be the half-sum of the roots in $N$, counted with their root-space dimensions. Equivalently, $2\rho_P$ is the character of $M$ on $\det(\operatorname{Lie}(N))$ (see \cite[Chapter 10]{GetzHahn} and \cite[Section I.1]{MoeglinWaldspurger}). Then
\[ 2\rho_P = (n-2)\varepsilon_1+(n-2)\varepsilon_2+(\varepsilon_1+\varepsilon_2) = (n-1)(\varepsilon_1+\varepsilon_2). \]
Hence,
\[ \rho_P=\frac{n-1}{2}(\varepsilon_1+\varepsilon_2). \]
Finally, for a connected reductive group $H$, its \emph{Langlands dual group} $\widehat H$ is the connected complex reductive group with root datum dual to the absolute root datum of $H$. In particular,
\[ \widehat G=\Sort_{n+2}(\C),\qquad \widehat M=\GL_2(\C)\times\Sort_{n-2}(\C). \]
Here one uses the absolute root datum, not the relative root system of type $B_2$ above (see \cite[Chapter 7]{GetzHahn} and \cite[Section 6.1.2]{ChenevierLannes}).

\subsection*{Automorphic forms on adele groups}

Recall that the adeles $\A$ over $\Q$ are defined by $\A = \R \times \A_\fin$ with $\A_\fin = \prod_p' \Q_p$, the restricted product over the $\Q_p$. The adelic points of $G$ can be identified with
\[  G(\A) = G(\R) \times \prod_p' G(\Q_p) . \]
The group $K_p\define G(\Z_p)$ for a prime $p$ is a compact open subgroup of $G(\Q_p)$, while $K_\infty\define\Stab_{G(\R)}(\R(e_1-f_1)\oplus\R(e_2-f_2))\cong\Sort_n\times\Sort_2$ is a maximal compact subgroup of $G(\R)$. We set $K_\fin \define\prod_pK_p$ and $K \define K_\infty\times K_\fin$. Then $K$ is a compact subgroup of $G(\A)$ and in good position with $P$ (see \cite[Appendix A.1]{GetzHahn}). By slight abuse of notation denote the complexification of the Lie algebra of $G(\R)$ by $\mathfrak{g}$.
Finally let $Z(\mathfrak{g})$ be the center of the universal enveloping algebra $U(\mathfrak{g})$ of $\mathfrak{g}$.

A smooth function
\[ \varphi: G(\A) \to \C \]
is called an \emph{automorphic form} on $G$ with respect to $K$ if it is
\begin{enumerate}[(i)]
	\item left $G(\Q)$-invariant,
	\item right $K$-finite,
	\item right $Z(\mathfrak{g})$-finite,
	\item of moderate growth.
\end{enumerate}
We denote the space of automorphic forms by $\mathcal{A}(G)$. We use the conventions of \cite[Chapter 6]{GetzHahn} and \cite[Section I.2]{MoeglinWaldspurger}. The definition of an automorphic form depends on $K_\infty$, but one can show that it does not depend on the choice of compact open subgroup $K_\fin$.

\medskip

Note that $M$ is also a reductive group with compact subgroup $M(\A)\cap K$. Analogously to $G$ we define the space $\mathcal{A}(M)$ of automorphic forms for $M$.

We define a character $\psi$ on $M(\A)$ by $\psi((m,\omega)) = |\det(m)|$, which extends to a character of $P(\A)$ and a function, but not a character on $G(\A)$ by letting it act trivially on $N(\A)$ and $K$. Note that here $|\cdot|$ refers to the norm on $\A$ and
\[ \psi = \psi_{\infty}\prod_{p}\psi_{p} \]
where $\psi_{\infty}((m_\infty,\omega_\infty)) = |\det(m_\infty)|_\infty$ and $\psi_{p}((m_p,\omega_p)) = |\det(m_p)|_p$. In particular, $\psi|_T = \varepsilon_1+\varepsilon_2$.

\subsection*{The hermitian symmetric space $G(\R)/K_{\infty}$}

The group $K_{\infty}$ is a maximal compact subgroup of $G(\R)$. It can be identified with $\text{SO}_n \times \text{SO}_2$. The quotient $G(\R)/K_{\infty}$ is a hermitian symmetric space of type $IV$. We denote by $\mfk$ the complexification of the Lie algebra of $K_\infty$. Then
\[ \mfg = \mfk \oplus \mfp_+\oplus\mfp_- \]
where $\mfp_+$ and $\mfp_-$ respectively can be identified with the holomorphic and anti-holomorphic tangent spaces on $G(\R)/K_\infty$ at the point $\id K_\infty$. The hermitian symmetric space $G(\R)/K_\infty$ has a projective model
\[ \mathbf{D} \define \{[Z]\in\mathbb{P}(V\otimes_\Q\C)\mid (Z,Z) = 0, (Z,\overline{Z}) < 0\}\cong G(\R)/K_\infty \]
and a tube domain model
\[ \widetilde{\mathcal{H}} \define\{Z\in V\otimes_\Q\C\cap\langle e_1,f_1\rangle^\bot\mid (\Ima(Z),\Ima(Z)) < 0\}\cong G(\R)/K_\infty. \]
Following the notation of Borcherds, we set $z\define e_1,z'\define-f_1$ and $\rho \define e_2,\rho'\define-f_2$. Then the isomorphism $\widetilde{\mathcal{H}}\to\mathbf{D}$ is given by
\[ Z \mapsto [Z + z' + \q(Z)z]. \]
This defines an action of $G(\R)$ on $\widetilde{\mathcal{H}}$ transported from the natural action of $G(\R)$ on $G(\R)/K_\infty$ and $\mathbf{D}$, namely for $g\in G(\R)$ we define a factor of automorphy $j(g,Z) \define -(g(Z+z'+\q(Z)z),z)$ and set $g Z\in\widetilde{\mathcal{H}}$ to be the unique element such that
\[ j(g,Z)\cdot(g Z + z' + \q(g Z)z) = g(Z + z' + \q(Z)z). \]
We parameterize
\begin{multline*}
	\widetilde{\mathcal{H}} = \{\tau'\rho + \tau\rho'+\mathfrak{z}\mid \\
	(\tau,\tau')\in(\mathbb{H}\times\mathbb{H})\cup-(\mathbb{H}\times\mathbb{H}), \mathfrak{z}\in W\otimes_\Q\C, \q(\Ima(\mathfrak{z})) < \Ima(\tau)\Ima(\tau') \}
\end{multline*}
and define the special point $Z_0\define i\rho+i\rho'$. Then $Z_0$ corresponds to $\id K_\infty$ in $G(\R)/K_\infty$ and $K_\infty = \Stab_{G(\R)}(Z_0)$.

For $m=\matr{a}{b}{c}{d}\in\GL_2(\R)$ we write $j(m,\tau)=\det(m)^{-1/2}(c\tau+d)$. We recall some useful formulas.

\begin{lem}\label{lem:ActionOfM}
	Let $g = (m,1)\in M(\R)$ with $m = \matr{a}{b}{c}{d}\in\GL_2(\R)$ and $Z = (\tau',\tau,\mathfrak{z})\in\widetilde{\mathcal{H}}$. Then
	\begin{align*}
gZ&=(\det(m)^{1/2}(\tau'-c\q(\mathfrak{z})j(m,\tau)^{-1}),m\tau,\det(m)^{1/2}j(m,\tau)^{-1}\mathfrak{z}),\\
j(g,Z)&=\frac{j(m,\tau)}{\det(m)^{1/2}}.
\end{align*} 
\end{lem}
\begin{proof}
Put $Q=\q(Z)=-\tau\tau'+\q(\mathfrak z)$. Direct computation gives
\begin{align*}
(m,1)(Z+z'+Qz)
&=(d\tau'-cQ)\rho+\frac{a\tau+b}{\det(m)}\rho'+\mathfrak z\\
&\qquad+\frac{c\tau+d}{\det(m)}z'+(aQ-b\tau')z.
\end{align*}
Dividing by the coefficient of $z'$ gives the asserted coordinates. The coefficient of $z$ is then $\q(gZ)$, since the vector is isotropic.
\end{proof}

Furthermore, we have

\begin{lem}\label{lem:qZisJ}
	Let $g\in G(\R)$ with $gZ_0 = Z$, then
	\[ -\q(\Ima(Z)) = |j(g,Z_0)|^{-2}. \]
\end{lem}
\begin{proof}
	We have
	\[ g(z+z'+i\rho+i\rho') = j(g,Z_0)(Z+z'+\q(Z)z). \]
	Therefore, 
	\begin{align*}
		-4 &= (z+z' + i\rho+i\rho',\overline{z+z'+i\rho+i\rho'}) \\
		&= (g(z+z' + i\rho+i\rho'),g(\overline{z+z'+i\rho+i\rho'})) \\
		&= |j(g,Z_0)|^2((Z,\overline{Z}) - 2\Rea(\q(Z))) \\
		&= 4|j(g,Z_0)|^2\q(\Ima(Z)).
	\end{align*}
\end{proof}

\begin{lem}\label{lem:characterIsImaginaryPart}
	Let $g\in G(\R)$ with $gZ_0 = Z$, then
	\[ \psi_\infty(g) = \frac{-\q(\Ima(Z))}{|(-\rho,\Ima(Z))|} = \frac{1}{|j(g,Z_0)|^{2}|(-\rho,\Ima(Z))|}. \]
\end{lem}	
\begin{proof}
	Write $g=(m,\omega)n(x_1,x_2,t)\kappa$ in Iwasawa form. We have
\[ \omega n(x_1,x_2,t)=n(\omega(x_1),\omega(x_2),t)\omega \]
and $\omega \kappa Z_0=Z_0$ so that without loss of generality we may assume $\omega \kappa = 1$. Put $u_0=z+z'+i\rho+i\rho'$. Then 
	\begin{align*}
		n(x_1,x_2,t)u_0
&= E^{z}_{t\rho+x_1}\bigl(z+z'+i\rho\\
&\qquad+i\rho'+ix_2+i\q(x_2)\rho\bigr) \\
		&= z+(z' + t\rho+x_1 +\q(x_1)z) +i\rho+(i\rho' - itz) + \\
		&\qquad+(ix_2 + i(x_1,x_2)z) + i\q(x_2)\rho \\
		&= (t+i(1+\q(x_2)))\rho + i\rho' + x_1 + ix_2 + z' + \q(Z_1)z
	\end{align*}
	with $Z_1 = n(x_1,x_2,t)Z_0 = (t+i(1+\q(x_2)))\rho + i\rho' + x_1 + ix_2$. Applying Lemma \ref{lem:ActionOfM} we obtain
	\begin{align*}
		Z &= (m,1)Z_1 \\
		&= \det(m)(t+i(1+\q(x_2)) - c\q(x_1+ix_2)j(m,i)^{-1})\rho + \\
		&\qquad+m i\rho' + \det(m)j(m,i)^{-1}(x_1+ix_2).
	\end{align*}
	and $j(g,Z_0) = j((m,1),Z_1) = j(m,i)\det(m)^{-1/2}$. Note that $\Ima(m i) = \frac{1}{|j(m,i)|^2}$ so that
	\begin{align*}
		|j(g,Z_0)|^2|(-\rho,\Ima(Z))| &= \frac{|j(m,i)|^2}{|\det(m)|}\cdot\frac{1}{|j(m,i)|^2} = \frac{1}{|\det(m)|}.
	\end{align*}
\end{proof}

\section{Hecke operators and constant terms}\label{sec:HeckeOperatorsAndConstantTerms}

In this section we recall some theory about Hecke operators and constant terms. The constructions apply to connected reductive groups, while the Satake description below is stated for split groups with hyperspecial level. These hypotheses hold for $G$ and $M$ over every $\Q_p$ in our unimodular setting. We normalize Haar measures by $\operatorname{vol}(K_p)=\operatorname{vol}(K_p\cap M)=\operatorname{vol}(K_p\cap N)=1$. Let
\[ \mathcal{H}_p(G) \define C_{\mathrm{c}}^\infty(K_p\backslash G(\Q_p)/K_p) \]
be the \emph{local Hecke algebra}, with convolution product and let
\[ \mathcal{H}(G) \define C_{\mathrm{c}}^\infty(K_\fin\backslash G(\A_\fin)/K_\fin)\cong\bigotimes_{p<\infty}'\mathcal{H}_p(G). \]
The restricted tensor product is taken with respect to $\mathbf{1}_{K_p}$. We denote by $\mathcal{A}(G,K_\fin)$ the space of automorphic forms of \emph{level $K_\fin$}, i.e.\ right $K_\fin$-invariant forms. For $f\in\mathcal{H}(G)$ its action is
\[ R(f)\varphi(x)=\int_{G(\A_\fin)}f(y)\varphi(xy)\mathrm{d}y. \]
We define $\mathcal{H}_p(M)$ and $\mathcal{H}(M)$ in the same way, with $K_{M,p}=M(\Q_p)\cap K_p$ and $K_{M,\fin}=\prod_pK_{M,p}$.

Let $\Rep(\widehat G)$ denote the complexified Grothendieck ring of finite-dimensional algebraic representations of $\widehat G$. The normalized \emph{Satake isomorphism} is
\[ \Sat:\mathcal{H}_p(G)\xrightarrow{\cong}\Rep(\widehat G). \]
The normalized constant term on the Hecke algebra is
\[ \iota_M^G(f)(m)=\psi_p(m)^{(n-1)/2}\int_{N(\Q_p)}f(mn)\mathrm{d}n. \]
It is an algebra homomorphism and the diagram
\[ \begin{tikzcd}
\mathcal{H}_p(G) \arrow[d,"\iota_M^G"] \arrow[rr,"\Sat"] && \Rep(\widehat G) \arrow[d,"\operatorname{Res}"] \\
\mathcal{H}_p(M) \arrow[rr,"\Sat"] && \Rep(\widehat M)
\end{tikzcd} \]
commutes (see \cite[Chapter 7]{GetzHahn} and \cite[Section 6.2]{ChenevierLannes}). For a non-split unramified group one must instead use the Satake description involving the Frobenius action on the dual group.

We normalize the invariant measure on $N(\Q)\backslash N(\A)$ to have volume $1$. The \emph{constant term} of $\varphi$ along $P$ is the function on $G(\A)$
\[ \varphi_P(g)=\int_{N(\Q)\backslash N(\A)}\varphi(ng)\mathrm{d}n. \]
Its restriction to $M(\A)$ is an automorphic form on $M$ (see \cite[Chapter 6]{GetzHahn} and \cite[Section I.2]{MoeglinWaldspurger}). If $\varphi$ has level $K_\fin$, then $\varphi_P$ has level $K_{M,\fin}$. For $f\in\mathcal{H}_p(G)$ we have
\begin{equation}\label{eq:HeckeCommutationRelationsAdelic}
(R(f)\varphi)_P=R\bigl(\psi_p^{-(n-1)/2}\iota_M^G(f)\bigr)\varphi_P.
\end{equation}
Indeed, right translation commutes with integration defining the constant term and Iwasawa decomposition of the integral over $G(\Q_p)$ gives the displayed formula. The expression $\psi_p^{-(n-1)/2}\iota_M^G(f)$ means pointwise multiplication of a compactly supported function by the character $\psi_p^{-(n-1)/2}$. The character itself does not belong to the Hecke algebra.

We will describe $\iota_M^G(T)$ explicitly for an example in Section \ref{sec:HolomorphicModularForms}.

\subsection*{Compact orthogonal groups}

Consider the genus $I\!I_{n-2,0}$ of positive-definite even unimodular of rank $n-2$. Let us denote by $\widetilde{I\!I}_{n-2,0}$ the set of isometry classes with respect to $\Sort_{W}$ of positive-definite even unimodular lattices of rank $n-2$, i.e.\ for each chiral lattice $[D]\in I\!I_{n-2,0}$, there are exactly two classes $[D^+],[D^-]\in\widetilde{I\!I}_{n-2,0}$ and one for each achiral lattice. We get a natural surjective map $\widetilde{I\!I}_{n-2,0}\to I\!I_{n-2,0}$. 

Put $K_W=\Sort_W(\widehat{\Z})$. The map
\[ \Sort_W(\Q)\backslash\Sort_W(\A)/(\Sort_W(\R)\times K_W)\longrightarrow\widetilde{I\!I}_{n-2,0},\quad
\omega\longmapsto[\omega_\fin(D_0\otimes\widehat{\Z})\cap W] \]
is a bijection. Hence, the space $\C[\widetilde{I\!I}_{n-2,0}]$ of formal sums $v = \sum_{[D]\in\widetilde{I\!I}_{n-2,0}}v_{[D]}[D]$ is identified with the space of automorphic forms on $\Sort_W$ with level $K_W$ and trivial $\Sort_W(\R)$-type. The vector $v$ corresponds to the automorphic form $\varphi_v(\omega)=v_{[\omega_\fin(D_0\otimes\widehat{\Z})\cap W]}$. Since $\Sort_W$ is anisotropic over $\Q$, its automorphic quotient is compact and every automorphic representation is cuspidal. Following \cite[Sections 4.1, 4.2 and 4.4]{ChenevierLannes}, we define
\[ \mathcal{H}(\Sort_W)=C_{\mathrm c}^\infty(K_W\backslash\Sort_W(\A_\fin)/K_W)=\bigotimes_p'\mathcal{H}_p(\Sort_W). \]
Then $\mathcal{H}(\Sort_W)$ acts on $\C[\widetilde{I\!I}_{n-2,0}]$ by the above identification.

Two even unimodular lattices $D,\widetilde{D}\subset W$ are \emph{$A$-neighbours}, for a finite abelian group $A$, if
\[ D/(D\cap \widetilde{D})\cong \widetilde{D}/(D\cap \widetilde{D})\cong A. \]
Let $\mathcal{N}_A([D])$ be the finite set of such $\widetilde{D}$ in the genus of $D$. For all chiral $[D]\in I\!I_{n-2,0}$ we choose once and for all a labeling $[D^+],[D^-]\in\widetilde{I\!I}_{n-2,0}$ for the two proper isomorphism classes and when convenient for achiral $D$ we will also use $[D^+] = [D^-] = [D]$. Then we may write $\mathcal{N}_A([D^{\pm}]) = \mathcal{N}_A^+([D^\pm])\cup\mathcal{N}_A^-([D^\pm])$ where $\mathcal{N}_A^+([D^\pm])$ contains only those $A$-neighbours $\widetilde{D}$ which are properly isomorphic to $[\widetilde{D}^\pm]$. We also Denote by $\mathcal{N}_A^*([D],[\widetilde{D}])$ the subset of $\mathcal{N}_A^*([D])$ containing only those lattices isomorphic to $\widetilde{D}$. Then the operator $T_A\in\mathcal{H}(\Sort_W)$ is characterized by
\[ (T_Av)_D=\sum_{\widetilde{D}\in\mathcal{N}_A^+([D])}v_{[\widetilde{D}]}. \]
Here the lattices are counted individually before passing to their proper isometry classes. At a prime $p$, this is the sum of the characteristic functions of the corresponding $K_{W,p}$-double cosets, with $\operatorname{vol}(K_{W,p})=1$. An orthogonal double coset may split into two special orthogonal double cosets and $T_A$ then denotes their sum.

Write $\Ort_W=\Ort(W)$ and let $K_{O,p}=\Ort(D_0\otimes\Z_p)$. Restriction from the orthogonal group gives the subalgebra
\[ \mathcal{H}(\Ort_W)=\bigotimes_p'C_{\mathrm c}^\infty(K_{O,p}\backslash\Ort_W(\Q_p)/K_{O,p})\ \hookrightarrow\ \mathcal{H}(\Sort_W). \]
Locally its image is the invariant subalgebra under the outer involution induced by an isometry of determinant $-1$ (see \cite[Section 4.2.11]{ChenevierLannes}). In particular, it contains every $T_A$. It preserves the two eigenspaces for reversal of orientation,
\[ \C[\widetilde{I\!I}_{n-2,0}]=\C[I\!I_{n-2,0}^{1}]\oplus\C[I\!I_{n-2,0}^{\det}]. \]

The Hecke action is simultaneously diagonalizable for the inner product
\[ \langle v,u\rangle=\sum_{[D]\in\widetilde{I\!I}_{n-2,0}}\frac{v_D\overline{u_D}}{|\Sort(D)|}. \]
A simultaneous eigenspace is the space of $\Sort_W(\R)\times K_W$-fixed vectors in the corresponding automorphic isotypic component. Its dimension is the automorphic multiplicity, since each unramified local representation has a one-dimensional spherical subspace. Thus a non-zero eigenvector $v$ determines an irreducible automorphic representation $\pi_v$, unramified at every finite place and trivial at infinity. These identifications and their compatibility with Hecke operators are discussed in \cite[Sections 4.3--4.4]{ChenevierLannes}.

\section{Eisenstein series}\label{sec:EisensteinSeries}

We study Eisenstein series for reductive groups in the framework of \cite{Langlands}. 

\begin{defthm}[{\cite[Chapter 10]{GetzHahn}, \cite[Sections II.1 and IV.1]{MoeglinWaldspurger}, \cite{Langlands}}]\label{thm:EisensteinSeriesDefinition}
	Let $\pi$ be a cuspidal automorphic representation of $M$ with unitary central character and let $\varphi_s\in\mathrm{Ind}_{MN}^G(\psi^{s+(n-1)/2}\otimes\pi\otimes1) \enifed I(\pi,s)$ be a holomorphic $K$-finite section of normalized parabolic induction. Then the \emph{Eisenstein series}
	\[ E(g,\varphi_s) \define \sum_{\gamma\in P(\Q)\backslash G(\Q)}\varphi_s(\gamma g),\qquad g\in G(\A) \]
	converges absolutely for $\Rea(s) > (n-1)/2$ to a function holomorphic in $s$ and has a meromorphic continuation to all $s\in\C$. Whenever $E(g,\varphi_s)$ does not have a pole, it defines an element in $\mathcal{A}(G)$.
\end{defthm}

We are interested in the case $\pi = \pi_f\otimes\pi_v$ where $\pi_f$ is the automorphic representation of $\GL_2$ associated to a simultaneous eigenform $f\in\mathrm{S}_k(\SL_2(\Z))$ of weight $k$ and $\pi_v$ is an automorphic representation of $\Sort_W$ of level $\Sort_W(\widehat{\Z})$ with trivial Archimedean component. As we have just seen, $\pi_v$ corresponds to an eigenvector $v\in\C[\widetilde{I\!I}_{n-2,0}]$. 

Let $u\in W(M,M)$, choose a representative in $G(\Q)$ and set
\[ N_u=N\cap uN^-u^{-1}, \]
where $N^-$ is the unipotent radical opposite to $N$. For $\Rea(s)$ sufficiently large the \emph{standard intertwining operator} is
\[ M(u,s)\varphi_s(g)=\int_{N_u(\A)}\varphi_s(u^{-1}ng)\mathrm{d}n. \]
It maps $I(\pi,s)$ to $I(u\pi,us)$, where $(u\pi)(m)=\pi(u^{-1}mu)$ and commutes with right translations and right differentiation. The integral continues meromorphically in $s$. These are the unnormalized intertwining operators for normalized induction. We use the compatible measures for which the global functional equations and composition laws hold (see \cite[Section II.1.6]{MoeglinWaldspurger} and \cite[Chapter 10]{GetzHahn}).

Let $w\in G(\Q)$ act trivially on $W$ and send the quadruple $(z,z',\rho,\rho')$ to $(-\rho',-\rho,z',z)$. It represents the non-trivial element of $W(M,M)$ and $N_w=N$. On the $\GL_2$ factor it acts, up to an inner automorphism, by the contragredient. Since $f$ has trivial central character, $w\pi\cong\pi$. We fix this identification. We are interested in the operator
\[ M(w,s)\varphi_s \]
at the point $-(n-k-1)/2$. For a factorizable section we can write
\[ \pi \cong \bigotimes_{p\leq\infty}\pi_p,\qquad I(\pi,s) = \bigotimes_{p\leq\infty}I(\pi_p,s),\qquad \varphi_s = \bigotimes_{p\leq\infty}\varphi_{p,s} \]
and
\[ M(w,s)\varphi_s = \bigotimes_{p\leq\infty}M_p(w,s)\varphi_{p,s}. \]
Since $w^2=t(-1,-1)$ acts trivially on these inducing data, the global functional equation gives
\[ 1 = M(w^2,s) = M(w,ws)M(w,s) = M(w,-s)M(w,s). \]
Our goal is to show that $M(w,s)\varphi_s$ vanishes at $s=-(n-k-1)/2$. On the one-dimensional line of sections considered below, this is equivalent to a pole of $M(w,s)\varphi_s$ at $s=(n-k-1)/2$. We compute at this latter point.

\begin{thm}[Langlands' constant term formula, cf.\ {\cite[Theorem 6.2.1]{Shahidi}} and \cite{Langlands}]\label{thm:LanglandsConstantTerm}
	Let $\pi$ and $\varphi_s$ be as above. As an identity of meromorphic families of functions on $G(\A)$,
	\[ E(\cdot,\varphi_s)_P(g)=\sum_{u\in W(M,M)}M(u,s)\varphi_s(g)=\varphi_s(g)+M(w,s)\varphi_s(g). \]
	The constant term along a proper rational parabolic not conjugate to $P$ is zero. If $Q=g_0Pg_0^{-1}$ with $g_0\in G(\Q)$, then
	\[ E(\cdot,\varphi_s)_Q(g)=\varphi_s(g_0^{-1}g)+M(w,s)\varphi_s(g_0^{-1}g). \]
\end{thm}

\subsection*{$p$-adic intertwining operators}
Let $G$ be split over $\Q_p$ with hyperspecial subgroup $K_p$ and let $\pi_p$ be an irreducible unramified representation of $G(\Q_p)$ with non-zero spherical vector $f_p\in\pi_p^{K_p}$. Then the eigenvalues of $f_p$ define an element of $\Hom(\mathcal{H}_p(G),\C)$. The \emph{Satake-parameter} $s_p$ of $\pi_p$ is the conjugacy class of semi-simple elements in $\widehat{G}(\C)$ such that 
\[ Tf_p = \tr(\Sat(T)(s_p))f_p \]
for all $T\in\mathcal{H}_p(G)$. Let $r:\widehat{G}(\C)\to\GL_m(\C)$ be a finite dimensional representation of $\widehat{G}(\C)$. Then the Euler product
\[ L(s,\pi,r) \define \prod_{p<\infty}L_p(s,\pi_p,r) \define \prod_{p<\infty}\det(1-p^{-s}r(s_p))^{-1} \]
is absolutely convergent for $\Rea(s)$ sufficiently large \cite[Chapter 12]{GetzHahn}. When $G$ is a classical group, such as $\GL_n$, $\Sp_{2n}$ or $\Sort_m$, then $\widehat{G}$ has a so-called \emph{standard representation} $\Std$. In the case of $\GL_n$, this is the identity, for $\Sort_W$, it is the inclusion
\[ \Std:\Sort_{n-2}(\C)\hookrightarrow\GL(W\otimes\C). \]

Let us consider the reductive group $M\cong\GL_2\times\Sort_W$ with automorphic representation $\pi = \pi_f\otimes\pi_v$ for an eigenform $f\in\mathrm{S}_k(\SL_2(\Z))$ and $\pi_v$ of level $\Sort_W(\widehat{\Z})$ with trivial Archimedean component. Since $\pi$ has level one, $\varphi_{p,s}$ is (up to normalization) the unique $K_p$-invariant vector in the unramified representation $I(\pi_p,s)$ and also  $M_p(w,s)\varphi_{p,s}$ must be a multiple of the unique (normalized) $K_p$-invariant vector $\varphi_{p,-s}\in I(\pi_p,-s)$. We denote the factor by $c_p(s)$ and define $c(s) \define\prod_{p<\infty}c_p(s)$. The local factor $c_p(s)$ is given by Euler-factors:

Let $r$ be the adjoint representation of $\widehat M$ on $\widehat{\mathfrak n}=\operatorname{Lie}(\widehat N)$. Its decomposition $r=\bigoplus_jr_j$ into irreducible subrepresentations is graded by the central cocharacter $a\mapsto(aI_2,1)$ of $\widehat M$, which acts on $r_j$ by $a^j$. The unramified intertwining formula gives
\[ c_p(s)=\prod_j\frac{L_p(js,\pi_p,r_j)}{L_p(js+1,\pi_p,r_j)},\qquad
c(s)=\prod_j\frac{L(js,\pi,r_j)}{L(js+1,\pi,r_j)}, \]
see \cite[Section 3.4]{Cogdell2}. This formula uses the absolute roots of the split local groups.

\begin{lem}\label{lem:DecompositionOfLeviAction}
	The representation $r:\widehat{M}\to\GL(\widehat{\mathfrak{n}})$ decomposes as $r = r_1\oplus r_2$ where
	\[ r_1 \cong\Std_{\GL_2}\otimes\Std_{\Sort_{n-2}},\qquad r_2\cong{\det}_{\GL_{2}}.  \]
\end{lem}
\begin{proof}
	Recall that 
	\[ \widehat{G}\cong \Sort_{n+2}(\C) \quad\text{and}\quad \widehat{M}\cong
	\GL_2(\C)\times \Sort_{n-2}(\C) \]
	and the Lie algebra $\widehat{\mathfrak{n}}$ of $\widehat{N}$ has the same two-step grading $\widehat{\mathfrak{n}}=\mathfrak{n}_1\oplus\mathfrak{n}_2$ as $\mathfrak{n}$ namely
	\[ \mathfrak{n}_1\cong \operatorname{Hom}(W,F)
	\cong F\otimes W,
	\qquad
	\mathfrak{n}_2\cong \Lambda^2F. \]
	Here we used $W^\vee\cong W$ via the quadratic form. Now $\widehat{M}(\C)$ acts on $X\in\operatorname{Hom}(W,F)$ via
	\[ (m,\omega)\cdot X = \gamma X\omega^{-1} \]
	and on $Z\in\Lambda^2F$ via
	\[ (m,\omega)\cdot Z=(\Lambda^2\gamma)Z. \]
	It follows from $\dim(F)=2$ that $\Lambda^2F\cong \det_{\GL_F}$ is one-dimensional. Therefore,
	\[ \mathfrak{n}_1 \cong\Std_{\GL_2}\otimes\Std_{\Sort_{n-2}},\qquad \mathfrak{n}_2\cong{\det}_{\GL_{2}}. \]
	The central cocharacter acts with weights $1$ and $2$ on these two spaces, which explains the indices $j=1,2$.
\end{proof}

Arthur associates to a representation $\pi\in\Pi_{\mathrm{disc}}(G)$ a formal symbol
\[ \psi(\pi,\Std) = \bigoplus_{i=1}^k\pi_i[d_i] \]
called its \emph{standard parameter}\label{sec:StandardParameters} where $\pi_i\in\Pi_{\mathrm{cusp}}(\GL_{n_i})$ are self-dual cuspidal automorphic representations, $d_i\in\Z_{\geq1}$ and $\sum_{i=1}^kn_id_i = m \define \dim(\Std)$.
Let $(s_\infty(\pi),s_2(\pi),s_3(\pi),\hdots)$ be the tuple indexed by $\infty$ and the primes where $s_p(\pi)\in\widehat{G}(\C)_{\mathrm{ss}}$ are the Satake-parameters of $\pi$ and $s_\infty(\pi)\in\widehat{\mathfrak{g}}_{\mathrm{ss}}$ is the image of the infinitesimal character of $\pi_\infty$ under the Harish--Chandra isomorphism. Note that $\widehat{\mathfrak{g}}$ denotes the Lie-algebra of $\widehat{G}(\C)$ and $X_{\mathrm{ss}}$ denotes the set of conjugacy classes of semi-simple elements in $X$. We may also define $[d]_p\in\SL_d(\C)_{\mathrm{ss}}$ as the element with eigenvalues $p^{-\frac{d-1}{2}+j}$ for $j=0,\hdots,d-1$ and $[d]_\infty\in\mathrm{Mat}_d(\C)_{\mathrm{ss}}$ as the element with eigenvalues $-\frac{d-1}{2}+j$ for $j=0,\hdots,d-1$. Then for any formal symbol $\psi$ as above we set
\[ \psi_p = \bigoplus_{i=1}^ks_p(\pi_i)\otimes[d_i]_p\in\GL_{m}(\C)_{\mathrm{ss}} \] 
At $p=\infty$ the tensor product is replaced by the sum of infinitesimal characters, i.e.\ the eigenvalues in a block are $w/2+j-(d_i-1)/2$. If $\psi$ is the standard parameter of $\pi$, then $\psi_\infty = \Std(s_\infty(\pi))$ and $\psi_p = \operatorname{Std}(s_p(\pi))$, respectively. Then the \emph{standard $L$-function} of $\pi$ is given by
\begin{align*}
	L(s,\pi,\Std) &\define\prod_{p<\infty}\det(1-p^{-s}\Std(s_p(\pi)))^{-1} \\
	&= \prod_{i=1}^k\prod_{j=0}^{d_i-1}L\left(s+j-\frac{d_i-1}{2},\pi_i\right)
\end{align*}
and the Euler product is absolutely convergent for $\Rea(s)>\frac{1+\max(d_i)_i}{2}$ following from the absolute convergence of $L(s,\pi_i,\Std)$ for $\Rea(s)>1$ (see \cite[Chapter 11]{GetzHahn} and \cite{Cogdell1}). The functions $L(s,\pi_i,\Std)$ admit holomorphic extensions to all of $\C$ \cite[Chapter 11]{GetzHahn} except when $\pi_i=1$ in which case $L(s,1,\Std) = \zeta(s)$ is the Riemann-$\zeta$-function, which is meromorphic on $\C$ with a simple pole at $s=1$.

\label{sec:MultiplicityFormulaDependencies}The multiplicity formula of Arthur describes the discrete spectrum of a reductive group in terms of standard parameters. Note that the multiplicity formula was conditional on the stabilization of the twisted trace formula for the groups $\GL_N$ and $\Sort_{2n}$ \cite[Hypothesis 3.2.1]{Arthur} as well as some results deferred to projected papers [A24]--[A27]. The issue corresponding to [A24] has since been resolved by M\oe glin--Waldspurger \cite{MoeglinWaldspurger1,MoeglinWaldspurger2} and in recent work by Atobe--Gan--Ichino--Kaletha--M\'inguez--Shin \cite{AGIKMS} the results of [A25]--[A27] have been established. These results leave Arthur's result conditional only on the twisted weighted fundamental lemma. Waldspurger has proved in \cite{Waldspurger} that this follows from the combination of the ordinary weighted fundamental lemma (WFL) and its non-standard version. Ordinary WFL for split groups was proved in \cite{ChaudouardLaumon1,ChaudouardLaumon2} so that the final missing piece is exactly (H1). In \cite{Shin} these dependencies are explained in more detail.

For $\pi\in\autrep(\GL_m)$, let $\mathrm{Weights}(\pi)$ denote its multiset of weights, i.e.\ the elements of $\mathrm{Weights}(\pi)$ are the eigenvalues of the conjugacy class of semi-simple elements of $\mathrm{Mat}_m(\C)$ corresponding to the infinitesimal character of $\pi_\infty$. In \cite{ChenevierLannes} the representations $\pi_v$ were studied in terms of their Arthur parameters. In particular, we have

\begin{lem}[\cite{ChenevierLannes}]\label{lem:ArthurParameterUnimodularProperties}
	Let $\psi=\bigoplus_{i=1}^m\pi_i[d_i]$ be the Arthur parameter of $\pi_v$. Then
	\begin{enumerate}[1)]
		\item the $\pi_i$ are algebraic and self-dual, in particular,
\[ -\mathrm{Weights}(\pi_i)=\mathrm{Weights}(\pi_i)\subset\tfrac{1}{2}\Z, \]
		\item $w\equiv d_i-1\bmod2$ for every $w/2\in\mathrm{Weights}(\pi_i)$,
		\item $|w|\leq n-d_i-3$ for every $w/2\in\mathrm{Weights}(\pi_i)$. If $w\ne0$, then also $d_i-1\leq |w|$,
		\item the multiset
		\[ \{w/2-(d_i-1)/2+j\mid i=1,\hdots,m, w/2\in\mathrm{Weights}(\pi_i), j=0,\hdots,d_i-1 \} \]
		contains the $n-4$ integers $\pm(n/2-2),\pm(n/2-3),\hdots,\pm1$ as well as the integer $0$ with multiplicity $2$. 
	\end{enumerate}
	We order $\psi$ such that $\max(\mathrm{Weights}(\pi_i)) + (d_i-1)/2$ is in ascending order.
\end{lem}
For two automorphic representations $\pi\in\autrep(\GL_m)$ and $\pi'\in\autrep(\GL_{m'})$ we set $L(s,\pi\otimes\pi') = L(s,\pi\otimes\pi',\Std_{\GL_m},\Std_{\GL_{m'}})$. For a suitable $\Gamma$-factor $L_\infty(s,\pi\otimes\pi')$, the function $\xi(s,\pi\otimes\pi') \define L_\infty(s,\pi\otimes\pi')L(s,\pi\otimes\pi')$ satisfies the functional equation
\begin{equation}\label{eq:FunctionalEquation}
	\xi(s,\pi\otimes\pi') = \varepsilon(\pi\otimes\pi')\xi(1-s,\pi^\vee\otimes\pi^{\vee\prime})
\end{equation}
where $\varepsilon(\pi\otimes\pi')\in\C^\times$ (cf.\ \cite{Cogdell1} for more information on $L$-functions).

\begin{lem}\label{lem:GammaFactorOfLfunctions}
	Let $f\in\mathrm{S}_k(\SL_2(\Z))$ and $\pi\in\autrep(\GL_m)$ algebraic and self-dual with multiset of weights $\mathrm{Weights}(\pi)$. Then
	\[L_\infty(s,\pi_f\otimes\pi) = \prod_{w/2\in\mathrm{Weights}(\pi)}\Gamma_\C\left(s+\frac{|w+k-1|}{2}\right) \]
\end{lem}
\begin{proof}
	\cite[Chapter 8.2]{ChenevierLannes}.
\end{proof}

We can now prove the main result of this section on the behaviour of $c(s)$ at $s=(n-k-1)/2$, which is

\begin{prp}\label{prp:ZerosAndPolesc}
	We assume $k\equiv n\bmod2$. Let $f\in\mathrm{S}_k(\SL_2(\Z))$ and let $\psi=\bigoplus_{i=1}^m\pi_i[d_i]$ be the Arthur parameter of $\pi_v$. Then $c(s)$ has a pole at $s = (n-k-1)/2$ if and only if one of the following conditions hold
	\begin{enumerate}[(i)]
		\item $k \geq n+2$
		\item $k=n$ and $L(\frac{1}{2},\pi_f\otimes\pi_i)=0$ for some $i$ with $d_i=1$
		\item $k = n-2$ and $L(\frac{1}{2},\pi_f\otimes\pi_i)\not=0$ whenever $d_i = 1$
		\item $k < n-2$, $\pi_i\cong\pi_f$ and $d_i=n-k-2$ for some $i$ and $L(\frac{1}{2},\pi_f\otimes\pi_i)\not=0$ whenever $d_i = n-k-1$.
	\end{enumerate}
\end{prp}
\begin{proof}
	By definition
	\[ L(s,\pi_f\otimes\pi_v,\Std_{\GL_2}\otimes\Std_{\Sort_{n-2}}) = \prod_{i=1}^m\prod_{j=0}^{d_i-1}L(s-(d_i-1)/2+j,\pi_f\otimes\pi_i) \]
	so that
	\begin{equation}\label{eq:LfunctionFactorStandardRep}
		\frac{L(s,\pi_f\otimes\pi_v,\Std_{\GL_2}\otimes\Std_{\Sort_{n-2}})}{L(s+1,\pi_f\otimes\pi_v,\Std_{\GL_2}\otimes\Std_{\Sort_{n-2}})} = \prod_{i=1}^m\frac{L(s-(d_i-1)/2,\pi_f\otimes\pi_i)}{L(s+(d_i+1)/2,\pi_f\otimes\pi_i)}.
	\end{equation}
	While
	\[ L(s,\pi_f\otimes\pi_v,{\det}_{\GL_2}) = \zeta(s). \]
	Since the Riemann-zeta function has a simple pole at $s=1$ and simple zeros at the negative even integers, the factor
	\[ \frac{L(2\cdot(n-k-1)/2,\pi_f\otimes\pi_v,{\det}_{\GL_2})}{L(2\cdot(n-k-1)/2+1,\pi_f\otimes\pi_v,{\det}_{\GL_2})} = \frac{\zeta(n-k-1)}{\zeta(n-k)} \]
	produces a simple pole if and only if $k\geq n+2$ or $k=n-2$. Therefore, we must show the following behaviour of the factor  \eqref{eq:LfunctionFactorStandardRep} at $s=(n-k-1)/2$: \\
	If $k\geq n-2$ it has a zero if and only if $k=n-2$ but (iii) is not satisfied and a pole if and only if (ii) is satisfied. If $k<n-2$ it has a pole if and only if case (iv) is satisfied. First note that the Euler product defining $L(s,\pi_f\otimes\pi_i)$ converges for $\Rea(s) > 1$, so that it has no zeros or poles in this sector. It satisfies the functional equation \eqref{eq:FunctionalEquation}, So if $\Rea(s) < 0$, then $L(s,\pi_f\otimes\pi_i)$ has no poles and any zeros come from poles of $L_\infty(s,\pi_f\otimes\pi_i)$. In the strip $0\leq \Rea(s)\leq1$, there can be non-trivial zeros and poles. Let us first address the trivial zeros: \\
	By Lemma \ref{lem:GammaFactorOfLfunctions}, the numerator of \eqref{eq:LfunctionFactorStandardRep} produces a zero for each weight $w/2\in\mathrm{Weights}(\pi_i)$ for which
	\[ (n-k-1)/2-(d_i-1)/2 \in \Z_{\leq0} - \frac{|w+k-1|}{2}. \]
	By Lemma \ref{lem:ArthurParameterUnimodularProperties}, $d_i-1\equiv w\bmod2$. So if $n\equiv 1\bmod2$ this is impossible. We assume that $n$ is even, then it is equivalent to the condition
	\[ n - k - d_i \leq -|w+k-1|. \]
	If $w\geq-(k-1)$, then
	\[ w \leq -(n-d_i-1) \leq -|w| -2 \]
	by Lemma \ref{lem:ArthurParameterUnimodularProperties}, which is a contradiction. If $w < -(k-1)<0$, then
	\[ w \geq n-d_i+1-2k \geq -w + 4 - 2k \]
	by Lemma \ref{lem:ArthurParameterUnimodularProperties}, hence $w\geq -(k-2)$, which is a contradiction. So the numerator of \eqref{eq:LfunctionFactorStandardRep} does not produce any additional zeros. Similarly, the denominator of \eqref{eq:LfunctionFactorStandardRep} produces a pole for each weight $w/2\in\mathrm{Weights}(\pi_i)$ for which
	\[ n - k + d_i \leq -|w+k-1|. \]
	But since $n - k + d_i > n - k - d_i$ this can not occur by the previous argument. The same gamma-factor inequalities exclude cancellation at the boundary of the critical strip.
	
	It remains to consider the critical strip $0\leq \Rea(s)\leq1$. For the numerator we consider
	\[ (n-k-1)/2-(d_i-1)/2 = t\in\{0,\textstyle\frac{1}{2},1\} \]
	or equivalently
	\[ d_i = n-k-2t. \]
	For the denominator we consider
	\[ (n-k-1)/2+(d_i+1)/2 = t'\in\{0,\textstyle\frac{1}{2},1\} \]
	or equivalently
	\[ k = n+d_i-2t'. \]
	First, assume that $k\geq n$. Then the numerator condition is impossible. We need to make sure that the denominator does not produce a zero, i.e.\ that $L(t',\pi_f\otimes\pi_i)$ does not have a pole. By \cite[Lecture 9, Section 3]{Cogdell1} this can only occur if $t'\in\{0,1\}$ and $\pi_i\cong\pi_f$. But then $(k-1)/2\in\mathrm{Weights}(\pi_i)$ so that
	\[ k-1 = n+d_i - 1 - 2t' = n-d_i-3 + 2(d_i-t'+1) > n-d_i-3 \]
	contradicting Lemma \ref{lem:ArthurParameterUnimodularProperties}. For $k=n$, the denominator has argument $1/2$ precisely when $d_i=1$. Thus it produces a pole of the quotient if and only if $L(\frac{1}{2},\pi_f\otimes\pi_i)=0$ for at least one such $i$. There are no canceling zeros in the numerator. 

	Now we assume $k\leq n-2$, then the denominator condition is impossible. Again by \cite[Lecture 9, Section 3]{Cogdell1}, the numerator produces a simple pole if and only if there exists an $i$ for which $\pi_i\cong\pi_f$ and $d_i = n-k-2t$ with $t\in\{0,1\}$. So again $(k-1)/2\in\mathrm{Weights}(\pi_i)$. If $t=0$, then 
	\[ k-1 = n-d_i-1 > n-d_i-3 \]
	contradicting Lemma \ref{lem:ArthurParameterUnimodularProperties} so only the case $d_i = n-k-2$ remains, which, in particular, also implies $k<n-2$. This pole is unique, since such a summand contains the extreme infinitesimal weight $n/2-2$, which occurs with multiplicity one. This is exactly case (iv). Finally, the numerator contributes a zero exactly when $t=\frac{1}{2}$ and $L(\frac{1}{2},\pi_f\otimes\pi_i)=0$.
\end{proof}

\subsection*{The Archimedean intertwining operator}

Let $\varphi_{\infty,s}\in\mathrm{Ind}_{M(\R)N(\R)}^{G(\R)}(\psi_{\infty}^{s+(n-1)/2}\otimes\pi_\infty\otimes1) \enifed I(\pi_\infty,s)$ be the Archimedean part of our section. In our case $\pi_\infty = \pi_k\otimes1$ where $\pi_k$ is the holomorphic discrete series representation of weight $k$ and
\[ \varphi_{\infty,s}(g) = j(g,Z_0)^{-k}\psi_\infty^{s+(n-k-1)/2}(g). \]
This is the case because $j(g,Z_0)^{-k}$ already contributes a factor of $\psi_\infty^{k/2}$ relative to the normalization of $\pi_k$ and the space of $K$-type $-k$ is one-dimensional in $I(\pi_\infty,s)$.

We need to compute the integral
\begin{equation}
	M_\infty(w,s)\varphi_{\infty,s}(g) = 
	\int_\R\int_{W(\R)}\int_{W(\R)} \varphi_{\infty,s}(w^{-1}n(x_1,x_2,t)g)\mathrm{d}x_1\mathrm{d}x_2\mathrm{d}t.
\end{equation}
This is well-defined for $\Rea(s)\gg0$ and has a meromorphic extension to all $s\in\C$. We are interested in the zeros and poles at the point $s=(n-k-1)/2$. Note that the subspace of $I(\pi_\infty,s)$ of $K$-type $-k$ has dimension $1$ and $M_\infty(w,s):I(\pi_\infty,s) \to I(\pi_\infty,-s)$. Therefore
\[ M_\infty(w,s)\varphi_{\infty,s} = \gamma(s)\varphi_{\infty,-s}\in I(\pi_\infty,-s) \]
for some meromorphic function $\gamma(s)$.

\begin{lem}
	For $t\in\R$ and $x_1,x_2\in W\otimes_\Q\R$ one has
	\[ \psi_\infty(w^{-1}n(x_1,x_2,t)) = t^2-t(x_1,x_2) + (1+\q(x_1))(1+\q(x_2)). \]
\end{lem}
\begin{proof}
	Let $Z \define w^{-1}Z_1 \define w^{-1}n(x_1,x_2,t)Z_0$ and we write $n$ for $n(x_1,x_2,t)$. As we have seen in the proof of Lemma \ref{lem:characterIsImaginaryPart}
	\[ Z_1 = (t+i(1+\q(x_2)))\rho + i\rho' + x_1 + ix_2\quad\text{and}\quad j(n,Z_0) = 1 \]
	and so
	\begin{align*}
		Z = j(w^{-1},Z_1)^{-1}(-\rho - \q(Z_1)\rho' + x_1 + ix_2)
	\end{align*}
	with $j(w^{-1}n,Z_0) = j(w^{-1},Z_1) = t+i(1+\q(x_2))$ and
	\begin{align*}
		\q(Z_1) &= 1+\q(x_2) - it + \q(x_1) - \q(x_2) + i(x_1,x_2) \\
		&= 1+\q(x_1) - i(t - (x_1,x_2)).
	\end{align*}
	Therefore
	\begin{align*}
		|j(w^{-1}n,Z_0)|^2(-\rho,\Ima(Z)) &= -|j(w^{-1}n,Z_0)|^2\Ima(j(w^{-1},Z_1)^{-1}\q(Z_1)) \\
		&= -\Ima((t-i(1+\q(x_2)))\cdot\q(Z_1)) \\
		&= t^2-t(x_1,x_2) + (1+\q(x_1))(1+\q(x_2)).
	\end{align*}
\end{proof}

By the previous lemma, $\gamma(s)$ is equal to the integral
\begin{align}\label{eq:TheFinalIntegral}
	\begin{split}
		\int_\R&\int_{W(\R)}\int_{W(\R)} (t+i(1+\q(x_2)))^{-k}\times \\
		&\times(t^2-t(x_1,x_2) + (1+\q(x_1))(1+\q(x_2)))^{-s -(n-k-1)/2}\mathrm{d}x_1\mathrm{d}x_2\mathrm{d}t \\
		&= \int_\R\int_{W(\R)}\int_{W(\R)} (t+i)^{-k}(\q(x_2)+1)^{-s-(n+k-3)/2}\times \\
		&\times(t^2(\q(x_2)+1)-t(x_1,x_2) + 1+\q(x_1))^{-s -(n-k-1)/2}\mathrm{d}x_1\mathrm{d}x_2\mathrm{d}t \\
		&= \int_\R\int_{W(\R)}\int_{W(\R)} (t+i)^{-k}(\q(x_2)+1)^{-s-(n+k-3)/2}\times \\
		&\times(\q(x_1) + t^2 + 1)^{-s -(n-k-1)/2}\mathrm{d}x_1\mathrm{d}x_2\mathrm{d}t \\
		&= \int_\R\int_{W(\R)}\int_{W(\R)} (t+i)^{-k}(t^2+1)^{-s+(k-1)/2}(\q(x_2)+1)^{-s-(n+k-3)/2}\times \\
		&\times(\q(x_1) + 1)^{-s -(n-k-1)/2}\mathrm{d}x_1\mathrm{d}x_2\mathrm{d}t \\
	\end{split}
\end{align}
where we made the substitutions $t\mapsto t(\q(x_2)+1)$, $x_1\mapsto x_1 + tx_2$ and $x_1\mapsto x_1\sqrt{t^2+1}$. The following is well-known and goes back to Euler.
\begin{lem}\label{lem:EulerFormulaIntegral}
	Let $A\in\mathrm{Sym}_d(\R)$ be positive-definite, $b\in\R^d$ and $c\in\R$ with $c-\frac14b^TA^{-1}b>0$. For $\Rea(s)>d/2$ we have
	\[ \int_{\R^d}(x^TAx + b^Tx + c)^{-s}\mathrm{d}x = \frac{\pi^{d/2}}{\sqrt{\det(A)}}\left(c-\frac{1}{4}b^TA^{-1}b\right)^{-s+d/2}\frac{\Gamma(s-d/2)}{\Gamma(s)}. \]
\end{lem}
\begin{proof}
Put $c_0=c-\frac14b^TA^{-1}b$. Completing the square and translating $x$ reduces the integral to the case $b=0$ and $c=c_0$. Substituting $t\mapsto t(x^TAx+c_0)$ yields
\begin{align*}
\Gamma(s)\int_{\R^d}(x^TAx+c_0)^{-s}\mathrm dx
&=\int_0^\infty t^{s-1}e^{-t}\int_{\R^d}(x^TAx+c_0)^{-s}\mathrm dx\,\mathrm dt\\
&=\int_0^\infty t^{s-1}e^{-tc_0}\int_{\R^d}e^{-tx^TAx}\mathrm dx\,\mathrm dt\\
&=\frac{\pi^{d/2}}{\sqrt{\det A}}\int_0^\infty e^{-tc_0}t^{s-d/2-1}\mathrm dt\\
&=\frac{\pi^{d/2}}{\sqrt{\det A}}c_0^{-s+d/2}\Gamma(s-d/2).
\end{align*}
\end{proof}
Furthermore, we have
\begin{lem}\label{lem:BetaIntegral}
	Let $k\in\Z$. For $\Rea(2s+k)>1$ we have
	\[ \int_\R(t+i)^{-k}(t^2+1)^{-s}\mathrm{d}t = i^{-k}\sqrt{\pi}\frac{\Gamma(s+(k-1)/2)\Gamma(s+k/2)}{\Gamma(s+k)\Gamma(s)}. \]
\end{lem}
\begin{proof}
	We substitute $t\mapsto\tan(\theta)$. Then
	\[ \tan(\theta)+i = \frac{ie^{-i\theta}}{\cos(\theta)},\qquad \tan^2(\theta)+1 = \frac{1}{\cos^2(\theta)}  \]
	and
	\[ \frac{\partial}{\partial\theta}\tan(\theta) = \frac{1}{\cos^2(\theta)}. \]
	So the integral becomes
	\[ i^{-k}\int_{-\pi/2}^{\pi/2}e^{ik\theta}\cos^{2s+k-2}(\theta)\mathrm{d}\theta = 2i^{-k}\int_{0}^{\pi/2}\cos(k\theta)\cos^{2s+k-2}(\theta)\mathrm{d}\theta \]
	where the equality holds because $\sin$ is antisymmetric and $\cos$ is symmetric so that the imaginary part of $e^{ik\theta}$ does not contribute to the integral. The last integral is a well-known formula for the reciprocal of the so-called beta function $B(z_1,z_2)$ (see \cite[\href{https://dlmf.nist.gov/5.12.E5}{(5.12.5)}]{DLMF}). Hence, our integral equals
	\begin{align*}
		i^{-k}\frac{\pi}{2^{2s+k-2}(2s+k-1)B(s+k,s)} &= i^{-k}\pi\frac{\Gamma(2s+k)}{2^{2s+k-2}(2s+k-1)\Gamma(s+k)\Gamma(s)} \\
		&= i^{-k}\pi\frac{\Gamma(2s+k-1)}{2^{2s+k-2}\Gamma(s+k)\Gamma(s)} \\
		&= i^{-k}\sqrt{\pi}\frac{\Gamma(s+(k-1)/2)\Gamma(s+k/2)}{\Gamma(s+k)\Gamma(s)}
	\end{align*}
\end{proof}

Both identities extend meromorphically in $s$. We finally obtain
\begin{thm}
	The operator $M_\infty(w,s)$ maps $\varphi_{\infty,s}$ to $\gamma(s)\varphi_{\infty,-s}$ with
	\begin{align*}
\gamma(s)&=2^{n-2}i^{-k}\pi^{n-3/2}\frac{\Gamma(s)\Gamma(s+1/2)}{\Gamma(s+(n+k-3)/2)\Gamma(s+(n-k-1)/2)}\\
&\qquad\times\frac{\Gamma(s+(k-1)/2)}{\Gamma(s+(k+1)/2)}.
\end{align*}
\end{thm}
\begin{proof}
	We use Lebesgue measure for which $D_0\subset W(\R)$ has covolume $1$ and $\Z\subset\R$ has covolume $1$. Since $\q(x)=\frac{1}{2}x^TSx$ with $\det(S)=1$, each of the two integrations over $W(\R)$ contributes the factor $2^{(n-2)/2}$. This is compatible with $\operatorname{vol}(N(\Q)\backslash N(\A))=1$ and $\operatorname{vol}(N(\Z_p))=1$. No by \eqref{eq:TheFinalIntegral} and Lemma \ref{lem:EulerFormulaIntegral} and \ref{lem:BetaIntegral}, $\gamma(s)$ equals
	\begin{multline*}
	2^{n-2}i^{-k}\pi^{n-3/2}
	\frac{\Gamma(s-(k-1)/2)}{\Gamma(s+(n-k-1)/2)}
	\frac{\Gamma(s+(k-1)/2)}{\Gamma(s+(n+k-3)/2)}\\
	\times\frac{\Gamma(s)\Gamma(s+1/2)}{\Gamma(s+(k+1)/2)\Gamma(s-(k-1)/2)}.
	\end{multline*}
\end{proof}

\begin{cor}\label{cor:ZerosAndPolesgamma}
	Let $k\geq2$. Then at $s = (n-k-1)/2$ the function $\gamma(s)$ has neither a zero nor a pole.
\end{cor}
\begin{proof}
	Since $\frac{\Gamma(s+(k-1)/2)}{\Gamma(s+(k+1)/2)} = \frac{1}{s+(k-1)/2}$ and we assume $n>2$, this produces no zero or pole at $s = (n-k-1)/2$. The other arguments in the numerator are 
	\begin{align*}
		&(n-k-1)/2  \\
		(n-k-1)/2+1/2 = &(n-k)/2
	\end{align*}
	producing a simple pole if $k \geq n-1$ and no pole otherwise. In the denominator we have
	\begin{align*}
		(n-k-1)/2+(n+k-3)/2 &= n-2 \\
		(n-k-1)/2+(n-k-1)/2 &= n-k-1
	\end{align*}
	producing a simple zero if $k\geq n-1$.
\end{proof}

\subsection*{The global intertwining operator}

We have seen in the previous sections that
\[ M(w,s)\varphi_s = m(s)\varphi_{-s} \]
for a meromorphic function $m(s) = \gamma(s)c(s)$ that satisfies
\[ m(-s) = m(s)^{-1}. \]
Now Proposition \ref{prp:ZerosAndPolesc} and Corollary \ref{cor:ZerosAndPolesgamma} imply
\begin{thm}\label{thm:mainThmeoremIntertwinging}
	Let $f\in\mathrm{S}_k(\SL_2(\Z))$ and let $\psi=\bigoplus_{i=1}^m\pi_i[d_i]$ be the Arthur parameter of $\pi_v$. Then $m(s)$ has a zero at $s = -(n-k-1)/2$ if and only if one of the following conditions hold
	\begin{enumerate}[(i)]
		\item $k \geq n+2$
		\item $k=n$ and $L(\frac{1}{2},\pi_f\otimes\pi_i)=0$ for some $i$ with $d_i=1$
		\item $k = n-2$ and $L(\frac{1}{2},\pi_f\otimes\pi_i)\not=0$ whenever $d_i = 1$
		\item $k < n-2$, $\pi_i\cong\pi_f$ and $d_i=n-k-2$ for some $i$ and $L(\frac{1}{2},\pi_f\otimes\pi_i)\not=0$ whenever $d_i = n-k-1$.
	\end{enumerate}
\end{thm}

The regularity of $E(g,\varphi_s)$ is already determined by the regularity of $E(g,\varphi_s)_P$.

\begin{prp}\label{prp:HolomorphicityConstantTerms}
	Let $\pi$ be a cuspidal automorphic representation of $M$ and let $\varphi_s\in I(\pi,s)$ be a section holomorphic near $s_0$. Suppose that $s\mapsto M(w,s)\varphi_s$ is holomorphic at $s_0$ and
	\[
	\left.M(w,s)\varphi_s\right|_{s=s_0}=0,
	\qquad R(X)\varphi_{s_0}=0
	\quad\text{for all }X\in\mfp_-.
	\]
	Then $E(g,\varphi_s)$ is regular at $s_0$ and
	\[
	R(X)E(g,\varphi_{s_0})=0
	\quad\text{for all }X\in\mfp_-.
	\]
\end{prp}

\begin{proof}
	Since $\pi$ is cuspidal and $P$ is maximal and conjugate to its opposite, the constant term of $E(\cdot,\varphi_s)$ along any proper rational parabolic not conjugate to $P$ vanishes, while for $Q=g_0Pg_0^{-1}$ with $g_0\in G(\Q)$ we have 
	\[ E(\cdot,\varphi_s)_Q(g) = E(\cdot,\varphi_s)_P(g_0^{-1}g). \]
	From the constant term formula Theorem \ref{thm:LanglandsConstantTerm} it follows that 
	\[ E(\cdot,\varphi_s)_P(g_0^{-1}g) = \varphi_s(g_0^{-1}g)+M(w,s)\varphi_s(g_0^{-1}g). \]
	Hence, all proper constant terms of $E(\cdot,\varphi_{s})$ are holomorphic
	at $s_0$. Unfolding against a cusp form, followed by meromorphic continuation, shows that $E(\cdot,\varphi_{s})$ and all its Laurent coefficients are orthogonal to cusp forms. If $E(\cdot,\varphi_{s})$ has a pole at $s_0$, all proper constant terms of its leading Laurent coefficient are zero and hence cuspidal. Since it is also orthogonal to cusp forms, it
	vanishes, which is a contradiction. Thus $E(\cdot,\varphi_{s})$ is regular at $s_0$. 
	
	By assumption, for $Q=g_0Pg_0^{-1}$ and
	$X\in\mfp_-$ we have
	\[ (R(X)E(\cdot,\varphi_{s_0}))_Q(g) = \left(R(X)\varphi_{s_0}(g_0^{-1}\cdot)\right)(g) = (R(X)\varphi_{s_0})(g_0^{-1}g) = 0 \]
	because right differentiation commutes with left translation. Since	the constant terms along all other proper rational parabolics vanish as well, $R(X)E(\cdot,\varphi_{s_0})$ is cuspidal. Finally, in the region of absolute convergence, differentiation under the sum gives
	\[ R(X)E(\cdot,\varphi_s)=E(\cdot,R(X)\varphi_s) \]
	and this identity extends meromorphically in $s$. Unfolding therefore shows orthogonality to all cusp forms. At $s=s_0$ the derivative is also cuspidal, so it vanishes.
\end{proof}

\section{Holomorphic modular forms}\label{sec:HolomorphicModularForms}

In this section we define holomorphic modular forms. Recall the Hermitian symmetric space $\widetilde{\mathcal{H}}$. It has two connected components $\widetilde{\mathcal{H}} = \mathcal{H}\cup-\mathcal{H}$ where
\begin{multline*}
	\mathcal{H} = \{\tau'\rho + \tau\rho'+\mathfrak{z}\mid \tau,\tau'\in\mathbb{H}, \mathfrak{z}\in W\otimes_\Q\C, \q(\Ima(\mathfrak{z})) < \Ima(\tau)\Ima(\tau') \}.
\end{multline*}
We denote the cone given by the imaginary part of $\mathcal{H}$ by $\mathcal{C}$. There exists a well-defined group homomorphism
\[ \spin:\Ort(V(\R))\to \R^\times/(\R^\times)^2\cong\{\pm1\} \]
called the \emph{spinor norm}. For any ring $R\subset\R$ we define
\begin{align*}
	\Ort^+(L\otimes R)&\define\Ort(L\otimes R)\cap\ker(\spin), \\ \Sort^+(L\otimes R)&\define\Sort(L\otimes R)\cap\ker(\spin).
\end{align*}
Then $\Ort^+(V(\R))$ fixes $\mathcal{H}$. For a function $F:\mathcal{H}\to\C$, an integer $k\in\Z$ and $\gamma\in\Ort^+(V(\R))$, we define
\[ F|_k[\gamma](Z) = j(\gamma,Z)^{-k}F(\gamma Z). \]
The discrete subgroup $\Ort^+(L)$ acts properly discontinuously on $\mathcal{H}$.

Let $\Gamma\subset\Ort^+(L)$ be a finite index subgroup and $\chi$ a character of $\Gamma$. A function $F:\mathcal{H}\to\C$ is called a \emph{modular form of weight $k$ and character $\chi$ for $\Gamma$} if 
\begin{enumerate}[(i)]
	\item $F|_k[\gamma] = \chi(\gamma)F$ for all $\gamma\in\Gamma$,
	\item $F$ is holomorphic on $\mathcal{H}$,
	\item $F$ is holomorphic at the cusps.
\end{enumerate}
Note that (iii) is automatically satisfied by the Koecher principle since we assume $n\geq3$. We denote the space of modular forms by $\mathrm{M}_k(\Gamma,\chi)$.

\medskip

We recall the notation for unimodular lattices. There exists an even unimodular lattice of signature $(n,2)$ if and only if $n\equiv2\bmod8$ and this lattice is unique up to isomorphism and denoted by $I\!I_{n,2}$. We denote $\Gamma_n \define\Ort^+(I\!I_{n,2})$ and $\SGamma_n \define\Sort^+(I\!I_{n,2})$. We consider the characters $1$ and $\det$ of $\Gamma_n$. One immediately obtains
\[ \mathrm{M}_k(\SGamma_n,1) = \mathrm{M}_k(\Gamma_n,1) \oplus \mathrm{M}_k(\Gamma_n,\det). \]
Furthermore, the element $-\id\in\SGamma_n$ acts trivially on $\mathcal{H}$ and $j(-\id,Z) = -1$, so that $\mathrm{M}_k(\SGamma_n,1) = \{0\}$ for $k\equiv1\bmod2$.

The modular variety $\Gamma\backslash\mathcal{H}$ is not compact. It can be compactified by adding finitely many $0$- and $1$-dimensional cusps. The $0$-dimensional cusps are parameterized by the $\Gamma$-orbits of the isotropic lines $I\subset V$, the $1$-dimensional cusps are parameterized by the $\Gamma$-orbits of the isotropic subspaces $F\subset V$ of dimension $2$. Since we assume that $L$ is unimodular, for any such line $I$, we have $(I^\bot\cap L)/(I\cap L) = I\!I_{n-1,1}$ and for any such space $F$, we have $(F^\bot\cap L)/(F\cap L) = D$ where $[D]\in I\!I_{n-2,0}$ is positive-definite even unimodular of rank $n-2$. If on the other hand $[D]\in I\!I_{n-2,0}$ is arbitrary, then $L \cong I\!I_{1,1}\oplus I\!I_{1,1}\oplus D$. There is only one $\Gamma$-orbit of isotropic lines $I$, so that $\Gamma\backslash\mathcal{H}$ has a unique $0$-dimensional cusp, which is the unique cusp of the $1$-dimensional cusps. The $\SGamma_n$-orbits of isotropic subspaces $F$ are parameterized by the classes $[D]\in \widetilde{I\!I}_{n-2,0}$. The $\Gamma_n$-orbits of isotropic subspaces $F$ are parameterized by the classes $[D]\in I\!I_{n-2,0}$. We choose primitive isotropic vectors $z,z'\in L$ with $(z,z') = -1$ and for each $[D]\in \widetilde{I\!I}_{n-2,0}$ we choose primitive isotropic vectors $\rho_D,\rho_D'\in L\cap\langle z,z'\rangle^\bot$ with $(\rho_D,\rho_D') = -1$ such that $L\cap\langle z,z'\rangle^\bot\cap\langle \rho_D,\rho_D'\rangle^\bot\cong D$ and we will henceforth assume that this equivalence is an equality We may assume that $\rho_D+\rho_D'\in\mathcal{C}$. Then $F_D \define\Q z\oplus\Q\rho_D$ and $F_D' \define\Q z'\oplus\Q\rho_D'$ are $2$-dimensional isotropic subspaces of $V$ dual to each other. We choose a fixed $[D_0]\in \widetilde{I\!I}_{n-2,0}$ and set $F = F_{D_0}$, $F' = F_{D_0}'$, which is compatible with the notation in Section \ref{sec:Definitions}. So $P = P_{D_0}$ for $P_D\define\Stab_{G}(F_D)$ with $P_D = M_DN_D$.

\begin{prp}
	The linear map
	\begin{align*}
		\Phi_{D}:\mathrm{M}_k(\SGamma_n,1)&\to\mathrm{M}_k(\SL_2(\Z)) \\
		\Phi_{D}(F)(\tau) &= \lim\limits_{\tau'\to i\infty}F(\tau'\rho_D+\tau \rho_D' + \mathfrak{z})
	\end{align*}
	is independent of $\mathfrak{z}\in D\otimes_\Z\C$ and the choice of $\rho_D$, $\rho_D'$ and is well-defined. We extend the maps $\Phi_D$ linearly to the map
	\begin{align*}
		\Phi:\mathrm{M}_k(\SGamma_n,1)&\to\mathrm{M}_k(\SL_2(\Z))\otimes_\C\C[\widetilde{I\!I}_{n-2,0}] \\
		F&\mapsto\sum_{[D]\in \widetilde{I\!I}_{n-2,0}}\Phi_D(F)\otimes[D]
	\end{align*}
	that we call the \emph{Siegel-$\Phi$-operator}.
\end{prp}
\begin{proof}
	The independence of $\mathfrak{z}$ will follow immediately from Proposition \ref{prp:FourierJacobiExpansion}. Let $\widetilde{\rho}_D$ and $\widetilde{\rho}_D'$ be another choice for $\rho_D$ and $\rho_D'$. Then there exists an automorphism $\gamma\in\SGamma_n$ such that $\gamma(\rho_D) = \widetilde{\rho}_D$ and $\gamma(\rho_D') = \widetilde{\rho}_D'$	so that
	\[ \lim\limits_{\tau'\to i\infty}F(\tau'\widetilde{\rho}_D+\tau \widetilde{\rho}_D') = \lim\limits_{\tau'\to i\infty}F|_k[\gamma](\tau'\rho_D+\tau \rho_D') = \lim\limits_{\tau'\to i\infty}F(\tau'\rho_D+\tau \rho_D'). \]
	The fact that $\Phi_{D}(F)$ transforms as a modular form of weight $k$ for $\SL_2(\Z)$ is now immediate from Lemma \ref{lem:ActionOfM}, holomorphicity and holomorphicity at the cusps follows from the holomorphicity of $F$.
\end{proof}
Note that
\[ \Phi(\mathrm{M}_k(\Gamma_n,\chi)) \subset \mathrm{M}_k(\SL_2(\Z))\otimes_\C\C[I\!I_{n-2,0}^{\chi}] \]
and $\C[I\!I_{n-2,0}^{1}]$ is the image of the natural inclusion $\C[I\!I_{n-2,0}]\hookrightarrow\C[\widetilde{I\!I}_{n-2,0}]$ and  $\C[I\!I_{n-2,0}^{\det}]\subset\C[\widetilde{I\!I}_{n-2,0}]$ is spanned by $[D^+]-[D^-]$ for $D\in I\!I_{n-2,0}$ chiral.

Let $\Gamma=\Gamma_{n}$ or $\Gamma = \SGamma_n$. We denote by
\[ \mathrm{S}_k(\Gamma,\chi) \define \mathrm{M}_k(\Gamma,\chi)\cap\ker(\Phi) \]
the space of \emph{cusp forms} and define 
\[ \mathcal{E}_k(\Gamma,\chi) \define\mathrm{M}_k(\Gamma,\chi)\cap\mathrm{S}_k(\Gamma,\chi)^\bot \]
with respect to the Petersson product. Let $\mathcal{E}^1_k(\Gamma,\chi)\subset\mathcal{E}_k(\Gamma,\chi)$ be the preimage of $\mathrm{S}_k(\SL_2(\Z))\otimes_\C\C[\widetilde{I\!I}_{n-2,0}]$ under $\Phi$ and $\mathcal{E}^0_k(\Gamma,\chi)\subset\mathcal{E}_k(\Gamma,\chi)$ be the preimage of $\C E_k\otimes_\C\C[\widetilde{I\!I}_{n-2,0}]$ under $\Phi$. It follows from the uniqueness of the $0$-dimensional cusp that
\[ \mathrm{M}_k(\Gamma,\chi) = \mathcal{E}^0_k(\Gamma,\chi) \oplus \mathcal{E}^1_k(\Gamma,\chi) \oplus \mathrm{S}_k(\Gamma,\chi) \] 
and $\mathcal{E}_k^0(\Gamma_n,\det) = \{0\}$ and $\Phi( \mathcal{E}^0_k(\Gamma_n,1)) \subset \C E_k\otimes\sum_{[D]\in\widetilde{I\!I}_{n-2,0}}[D]$, so that $\dim(\mathcal{E}^0_k(\Gamma_n,1))\leq1$. We will see that the Eisenstein series studied in this article describe the space $\mathcal{E}^1_k(\Gamma,\chi)$.

The weight $n/2-1$ is called \emph{singular weight}. It is the smallest possible weight for which there can exist holomorphic modular forms and any modular form of singular weight lies in $\mathcal{E}_k(\Gamma,\chi)$. This follows from the discussion on Fourier--Jacobi expansions below.

\medskip

We give some examples of orthogonal modular forms. In \cite{Borcherds98}, Borcherds described two singular theta correspondences between symplectic modular forms and orthogonal modular forms. The first is additive and goes back to Oda, Rallis--Schiffmann and Gritsenko (\cite{Oda,RallisSchiffmann,GritsenkoFourierJacobi}). We denote it by $\Psi_L^{\mathrm{add}}$. It is a linear map from modular forms of weight $k$ for the Weil representation $\omega_L$ (see below) to orthogonal modular forms of weight $k + n/2-1$ for the orthogonal group of some even lattice $L$ of signature $(n,2)$. In particular, when $L=I\!I_{n,2}$ is unimodular, then
\[ \Psi_{I\!I_{n,2}}^{\mathrm{add}}:\mathrm{M}_k(\SL_2(\Z))\to\mathrm{M}_{k+n/2-1}(\Gamma_n,1). \]
It maps cusp forms to cusp forms and $\Psi_{I\!I_{n,2}}^{\mathrm{add}}(E_k)\in \mathcal{E}^0_{k+n/2-1}(\Gamma_n,1))$ for the unique Eisenstein series $E_k\in\mathrm{M}_k(\SL_2(\Z))$. Again by considering Fourier--Jacobi expansions, one finds that $\mathcal{E}_{n/2+1}(\Gamma,\chi) = \{0\}$. Therefore, $\mathcal{E}^0_{k}(\Gamma,\chi))$ can be described entirely by $\Psi_{I\!I_{n,2}}^{\mathrm{add}}$, indeed
\[ \mathcal{E}^0_{k}(\Gamma,\chi)) = \begin{cases}
	\C\Psi_{I\!I_{n,2}}^{\mathrm{add}}(E_{k-n/2+1}) & \substack{\text{if }\Gamma = \Gamma_n, \chi = 1, k\geq n/2-1 \\ k\not=n/2+1, k\equiv0\bmod2} \\
	\{0\} & \text{else}.
\end{cases} \]
The space $\mathcal{E}_k^1(\Gamma,\chi)$ can not be reached by the additive theta lift (cf.\ \cite[Theorem 14.3]{Borcherds98}).

The second lift described in \cite{Borcherds98} is multiplicative and we denote it by $\Psi_L^{\mathrm{mult}}$. It maps weakly holomorphic modular forms $f$ of weight $1-n/2$ for $\omega_{L}$ to meromorphic orthogonal modular forms of weight $c_f(0,0)/2$ (\cite[Theorem 13.3]{Borcherds98}). Its description is more complicated. As an example we consider
\[ \Phi_{12} \define \Psi_{I\!I_{26,2}}^{\mathrm{mult}}(1/\Delta)\in\mathrm{M}_{12}(\Gamma_{26},\det) \]
where
\[ \Delta(\tau) = e^{2\pi i\tau}\prod_{n=1}^\infty(1-e^{2\pi in\tau})^{24} = \sum_{n=1}^{\infty}\tau(p)e^{2\pi in\tau}\in\mathrm{S}_{12}(\SL_2(\Z)) \]
denotes the discriminant modular form. It is the unique holomorphic automorphic product of singular weight on a unimodular lattice (see \cite{ScheithauerCompositio}). As we have explained above, $\mathcal{E}_{12}^0(\Gamma_{26},\det) = \{0\} = \mathrm{S}_{12}(\Gamma_{26},\det)$, so that $\Phi_{12}\in\mathcal{E}_{12}^1(\Gamma_{26},\det)$. The genus $I\!I_{24,0}$ consists of the $24$ Niemeier lattices. Let $\Lambda$ be the Leech lattice, which is the unique Niemeier lattices without roots and hence, the unique Niemeier lattices which is chiral. It follows that
\[ \Phi(\Phi_{12}) = \Delta\otimes([\Lambda^+] - [\Lambda^-]) \]
up to normalization. We will show in Theorem \ref{thm:Phi12IsEisensteinSeries} that $\Phi_{12}$ can be described as an Eisenstein series.

\subsection*{Fourier--Jacobi expansions of orthogonal modular forms}

We recall the definition of Jacobi forms of lattice index, which are generalizations of classical Jacobi forms \cite{EichlerZagier} introduced by Gritsenko \cite{GritsenkoFourierJacobi} and of the Weil representation (cf.\ \cite{Weil,Borcherds98,ScheithauerIMRN}).

Let $L$ be a positive-definite even lattice of even rank. Then $\SL_2(\Z)$ acts from the right on the pairs $(\lambda,\mu) \in L \times L$. The corresponding semidirect product $J_L = \SL_2(\Z) \ltimes (L \times L)$ is the Jacobi group of lattice index $L$. Recall that the product of two elements in $J_L$ is given by
\[ \big( M_1,(\lambda_1,\mu_1) \big)
\big( M_2,(\lambda_2,\mu_2) \big) = \big( M_1M_2,(\lambda_1,\mu_1)M_2 + (\lambda_2,\mu_2) \big). \]
We identify $\SL_2(\Z)$ with the subgroup $\SL_2(\Z) \ltimes (0 \times 0)$ and $L\times L$ with $1\ltimes (L\times L)$. Now let $k \in \Z$. We define an action of the Jacobi group $J_L$ on the functions
$\phi: \mathbb{H} \times (L\otimes_\Z\C)\rightarrow\C$ by
\begin{align*}
	\phi|_k[M](\tau,\mathfrak{z})
	&= \phi \bigg(M\tau,\frac{\mathfrak{z}}{c\tau+d}\bigg) (c\tau+d)^{-k} e \bigg( \frac{-c\q(\mathfrak{z})}{c\tau+d} \bigg) \\
	\phi|_k[(\lambda,\mu)](\tau,\mathfrak{z})
	&= \phi\left(\tau,\mathfrak{z} + \lambda\tau+\mu\right) e\left(\tau \q(\lambda) + (\lambda,\mathfrak{z})\right), 
\end{align*}
where $M = \left(\begin{smallmatrix} a & b \\ c & d \end{smallmatrix}\right) \in \SL_2(\Z)$, $\lambda,\mu \in L$ and $e(z) = e^{2\pi i z}$. A {\em Jacobi form} of weight $k$ and index $L$ is a holomorphic function $\phi : \mathbb{H} \times (L\otimes_\Z\C) \to \C$ which is invariant under the action of $J_L$ and possesses a Fourier expansion of the form
\[ \phi(\tau,\mathfrak{z})
= \sumstack{m \in \Z , \lambda \in L' \\ m \geq \q(\lambda)} c_\phi(m,\lambda) e(m \tau + (\lambda,\mathfrak{z}))  \]
where $L'$ denotes the dual lattice of $L$. We denote the space of Jacobi forms of weight $k$ and lattice index $L$ by $J_{k,L}$.

There exists a representation $\omega_L:\SL_2(\Z)\to\GL(\C[L'/L])$ called the \emph{Weil representation} (cf.\ \cite{Weil, Borcherds98,ScheithauerIMRN}) which can be defined by
\begin{align*}
	\omega_L(T)\mathfrak{e}_{\alpha} &= e(-\q(\alpha))\mathfrak{e}_{\alpha} \\
	\omega_L(S)\mathfrak{e}_{\alpha} &= \frac{e(\rk(L)/8)}{\sqrt{|L'/L|}}\sum_{\beta\in L'/L}e((\beta,\alpha))\mathfrak{e}_{\beta}
\end{align*}
for $T = \matr{1}{1}{0}{1}, S = \matr{0}{-1}{1}{0}\in\SL_2(\Z)$ and $\alpha\in L'/L$. The Weil representation is unitary with respect to the scalar product $\langle\cdot,\cdot\rangle$ defined by
\[ \langle\mathfrak{e}_\alpha,\mathfrak{e}_\beta\rangle = \delta_{\alpha\beta}. \]
A \emph{modular form of weight $k$ for the Weil representation} is a holomorphic function $f:\mathbb{H}\to\C[L'/L]$ such that
\[ f(M\tau) = j(M,\tau)^k\omega_L(M)f(\tau) \]
and that possesses a Fourier expansion of the form
\[ f(\tau) = \sum_{\alpha\in L'/L}\sumstack{m\in\Z-\q(\alpha) \\ m\geq0}c_f(m,\alpha)e(m\tau)\mathfrak{e}_{\alpha}. \]
We denote the space of modular forms of weight $k$ for the Weil representation by $\mathrm{M}_k(\omega_L)$ and remark that $\mathrm{M}_0(\omega_L) =  \C[L'/L]^{\SL_2(\Z)}$ consists of the invariants of $\omega_L$.

A Jacobi form $\phi \in J_{k,L}$ has a unique theta expansion
\[  \phi(\tau,\mathfrak{z}) = \sum_{\alpha \in L'/L} f_{\alpha}(\tau)\vartheta_{\alpha}(\tau,\mathfrak{z}) ,  \]
where 
\[  \vartheta_{\alpha}(\tau,\mathfrak{z}) = \sum_{\lambda \in \alpha + L} e( \tau (\lambda,\lambda)/2 + (\lambda,\mathfrak{z}) )  \]
is the Jacobi theta function of the coset $\alpha + L$ and $f(\tau) = \sum_{\alpha \in L'/L} f_{\alpha}(\tau)\mathfrak{e}_\alpha$ is a modular form of weight $k-\rk(L)$ for the Weil representation $\omega_{L}$ with Fourier coefficients
\[ c_f(m,\lambda+L) = c_\phi(m+\q(\lambda),\lambda), \]
which are independent of the choice of coset representative by the periodicity of $\phi$.
We obtain a map
\[  J_{k,L} \to\mathrm{M}_{k-\rk(L)/2}(\omega_L),  \]
which is actually an isomorphism. This implies that $J_{k,L}$ is trivial for $k<n/2$. The weight $k=n/2$ is called {\em singular} weight.

Let $F\in\mathrm{M}_k(\SGamma_n,1)$. Then $F$ has a \emph{Fourier expansion} of the form
\[ F(Z) = \sum_{\lambda\in I\!I_{n-1,1} \cap \mathcal{C}}c_F(\lambda)e(-(\lambda,Z)) \]
following from the invariance under elements of the form $E^z_{\lambda}$ for $\lambda\in I\!I_{n-1,1}$. Recall that for $Z\in\mathcal{H}$ we may write $Z = \tau'\rho_D+\tau\rho_D'+\mathfrak{z}$ with $\tau,\tau'\in\mathbb{H}$ and $\mathfrak{z}\in D\otimes_\Z\C$. Then
\begin{align*}
	F(Z) &= \sumstack{a,b\in\Z_{\geq0} \\ \lambda\in D \\ \q(\lambda)\leq ab}c_F(a\rho_D+b\rho_D'+\lambda)e(a\tau)e(-(\lambda,\mathfrak{z}))e(b\tau') \\
	&= f_0(\tau) + \sum_{b=1}^\infty\phi_b(\tau,\mathfrak{z})e(b\tau')
\end{align*}
where
\[ f_0(\tau) = \sum_{a=0}^\infty c_F(a\rho_D)e(a\tau),\quad \phi_b(\tau,\mathfrak{z}) = \sumstack{a\in\Z_{\geq0} \\ \lambda\in D \\ \q(\lambda)/b\leq a}c_F(a\rho_D+b\rho_D'+\lambda)e(a\tau-(\lambda,\mathfrak{z})). \]
Denote by $D(b)\subset D\otimes_\Z\R$ the lattice $\sqrt{b}\cdot D$.
\begin{prp}\label{prp:FourierJacobiExpansion}
	Let $F\in\mathrm{M}_k(\SGamma_n,1)$ with notation as above then $f_0 = \Phi_D(F)\in\mathrm{M}_k(\SL_2(\Z))$ and $\phi_b(\tau,\frac{1}{\sqrt{b}}\mathfrak{z})\in J_{k,D(b)}$.
\end{prp}
\begin{proof}
	The fact that $f_0 = \Phi_D(F)$ follows immediately from the Fourier--Jacobi expansion. By Lemma \ref{lem:ActionOfM}, the action of $\SL_2(\Z)\ltimes (0 \times 0)\subset J_{D(b)}$ on $\phi_b$ is exactly given by the action of the Levi subgroup $M_D(\Z)$ on $F$ and the action of $1\ltimes (D(b) \times D(b))\subset J_{D(b)}$ by that of $N_D$.
\end{proof}

Since $D$ is unimodular, we have $D(b)' = D'(1/b) = D(1/b)$ and $D(b)'/D(b)\cong D/bD$. Let
\[ f_b(\tau) = \sum_{\lambda+bD\in D/bD}\sumstack{m\in\Z-\q(\lambda)/b \\ m\geq0}c_F((m+\q(\lambda)/b)\rho_D + b\rho_D'+\lambda)e(m\tau)\mathfrak{e}_{\frac{1}{\sqrt{b}}\lambda + D(b)} \]
be the theta decomposition of $\phi_b$. The Fourier coefficients for $m=0$ can be determined from $\Phi(F)$. Indeed, in this case $\q(\lambda)/b\rho_D + b\rho_D'+\lambda$ is isotropic, hence, there exists $[\widetilde{D}]\in \widetilde{I\!I}_{n-2,0}$, $a\in\Z_{\geq0}$ and $\gamma\in\Sort(I\!I_{n-1,1})$ such that $\gamma(a\rho_{\widetilde{D}}) = \q(\lambda)/b\rho_D + b\rho_D'+\lambda$. The modularity of $F$ under $\gamma$ implies that
\begin{equation}\label{eq:FourierCoefficientsSingular}
	c_F(\q(\lambda)/b\rho_D + b\rho_D'+\lambda) = c_F(a\rho_{\widetilde{D}}) = c_{\Phi_{\widetilde{D}}(F)}(a).
\end{equation}
When $F$ is of singular weight, then $\phi_b$ is of singular weight, so that $f_b\in\C[D(b)'/D(b)]^{\SL_2(\Z)}$ and only terms with $m=0$ occur.

\subsection*{Neighbour numbers}

In \cite{ChenevierLannes} the spaces $\C[I\!I_{n-2,0}^{1}]$ and $\C[I\!I_{n-2,0}^{\det}]$ are determined for $n-2 = 8,16,24$ in terms of the Arthur parameters of the corresponding automorphic representations using the multiplicity formula. This yields a new proof of the classification of the positive-definite even unimodular lattices in ranks $8,16$ and $24$ together with formulas for the Kneser-neighbour numbers. In particular, they show
\begin{prp}[{\cite[Proposition 7.5.1 and Corollary 5.3.2]{ChenevierLannes}}]
	The space $\C[I\!I_{24,0}^{\det}]$ is spanned by $[\Lambda^+] - [\Lambda^-]$ and the Arthur parameter of $\pi_{[\Lambda^+] - [\Lambda^-]}$ is $\pi_\Delta[12]$. Consequently,
	\begin{equation}\label{eq:NeighbourFormula}
		\#\mathcal{N}_{\Z/p\Z}^+([\Lambda],[\Lambda]) - \#\mathcal{N}_{\Z/p\Z}^-([\Lambda],[\Lambda]) = \tau(p)\frac{p^{12}-1}{p-1}.
	\end{equation}
\end{prp}
We give an entirely new proof of formula \eqref{eq:NeighbourFormula} using orthogonal modular forms of singular weight.

\begin{thm}\label{thm:NeighbourFormulaUsingOrthogonalModularForms}
	Let $F\in\mathrm{M}_{n/2-1}(\Gamma,\chi)$ and suppose that
	\[ \Phi(F) = f\otimes v \]
	for some $f\in\mathrm{M}_{n/2-1}(\SL_2(\Z))$ with $c_f(1) = 1$ and $v\in\C[I\!I_{n-2,0}^\chi]$. Then
	\[ T_{\Z/p\Z}v = c_f(p)\frac{p^{n/2-1}-1}{p-1}\cdot v. \]
\end{thm}
\begin{proof}
	Let $[D]\in \widetilde{I\!I}_{n-2,0}$. For $a,b\in\Z$, $\lambda\in D$ we define 
	\[ \gcd(a,b,\lambda) \define \max\{d\in\Z_{\geq0}\mid d\mu = a\rho_D+b\rho_D'+\lambda\text{ for some }\mu\in I\!I_{n-1,1} \} \]
	and $\gcd(a,b,\lambda) = -1$ if $a\in\Q\setminus\Z$. Furthermore, we set $u_\lambda\define\q(\lambda)/b\rho_D + b\rho_D'+\lambda$ and
	\[ D_\lambda \define u^\bot/u\cap I\!I_{n-1,0}. \]
	If $b|\q(\lambda)$, there exists a $\gamma\in\Sort(I\!I_{n-1,1})$ such that $\gcd(\q(\lambda)/b,b,\lambda)\gamma(\rho_{D_\lambda}) = u_\lambda$, which determines a unique isometry class $[D_\lambda]\in\widetilde{I\!I}_{n-2,0}$. Hence, by \eqref{eq:FourierCoefficientsSingular} the theta decomposition of the $b$-th Fourier--Jacobi coefficient of $F$ is given by
	\[ f_b = \sum_{\lambda+bD\in D/bD}v_{[D_\lambda]}c_f(\gcd(\q(\lambda)/b,b,\lambda))\mathfrak{e}_{\frac{1}{\sqrt{b}}\lambda + D(b)}\in\C[D(b)'/D(b)]^{\SL_2(\Z)} \]
	where $v = \sum_{[D]\in\widetilde{I\!I}_{n-2,0}}v_{[D]}[D]$. We write $f_{b,\lambda+bD}$ for the coefficient before $\mathfrak{e}_{\frac{1}{\sqrt{b}}\lambda + D(b)}$. Let us now consider $b=p$ prime. Then $\gcd(\q(\lambda)/p,p,\lambda)\in\{1,p\}$. If it is equal to $p$, then $\lambda\in pD$. If on the other hand $\lambda\in pD$, then $\q(\lambda)/p\in p\Z$. It follows that
	\[ f_{b,\lambda+bD} = \begin{cases}
		v_{[D]}c_f(p) & \text{if } \lambda\in pD \\
		v_{[D_\lambda]} & \text{if } \lambda\not\in pD\text{ and }\q(\lambda)\in p\Z \\
		0 & \text{else}.
	\end{cases} \]
	We need to determine $D_\lambda$. Let $v = x\rho_D+y\rho_D'+\mu\in I\!I_{n-1,1}$ with $(v,u_\lambda) = 0$. Then
	\[ -px - \q(\lambda)y/p + (\lambda,\mu) = 0 \]
	and hence
	\[ x = -\q(\lambda)y/p^2 + (\lambda,\mu)/p. \]
	One verifies that
	\begin{align*}
		\nu:u_\lambda^\bot&\to D\otimes_\Z\Q \\
		x\rho_D+y\rho_D'+\mu&\mapsto y/p\lambda-\mu
	\end{align*}
	defines a linear transformation with kernel $\langle u_\lambda\rangle$ that satisfies $\q(\nu(v)) = \q(v)$. Therefore, $N_{\lambda/p} \define \nu(u_\lambda^\bot\cap I\!I_{n-1,1})\cong D_\lambda$. By choosing a different representative in $\lambda+pD$ we may assume that $\q(\lambda)\in p^2\Z$. Since $x$ is integral,
	\[ N_{\lambda/p} = \frac{1}{p}\Z\lambda + \{\mu\in D\mid (\mu,\lambda)\in p\Z\}. \]
	Hence, $N_{\lambda/p}$ is a $p$-neighbour of $D$. Furthermore, any $p$-neighbour of $D$ can be obtained in this way and $N_{\lambda/p} = N_{\lambda'/p}$ if and only if $\lambda' = a\gamma+p\mu$ for some  $a\in(\Z/p\Z)^\times$ and $\mu\in D$ (see \cite[Section 1.2]{ChenevierLannes}).	Therefore, using the invariance of $f_p$ under $S = \matr{0}{-1}{1}{0}$ we obtain
	\begin{align*}
		v_{[D]}c_f(p) &= \langle\omega_{D(p)}(S)f_p,\mathfrak{e}_{D(b)}\rangle \\
		&= \frac{1}{p^{n/2-1}}\sum_{\alpha\in D/pD}f_{p,\alpha} \\
		&= \frac{1}{p^{n/2-1}}\left(v_{[D]}c_f(p) + (p-1)\sum_{[\widetilde{D}]\in\widetilde{I\!I}_{n-2,0}}v_{[\widetilde{D}]}\#\mathcal{N}_{\Z/p\Z}^+([D],[\widetilde{D}])\right)
	\end{align*}
	and hence
	\[ \sum_{[\widetilde{D}]\in\widetilde{I\!I}_{n-2,0}}v_{[\widetilde{D}]}\#\mathcal{N}_{\Z/p\Z}^+([D],[\widetilde{D}])) = v_{[D]}c_f(p)\frac{p^{n/2-1}-1}{p-1}. \]
	The left-hand-side is precisely $(T_{\Z/p\Z}v)_{[D]}$ and the right-hand-side is $c_f(p)\frac{p^{n/2-1}-1}{p-1}\cdot v_{[D]}$.
\end{proof}

Now \eqref{eq:NeighbourFormula} immediately follows from Theorem \ref{thm:NeighbourFormulaUsingOrthogonalModularForms} by considering $F = \Phi_{12}$ and the fact that $\Phi(\Phi_{12}) = \Delta\otimes([\Lambda^+] - [\Lambda^-])$.

Theorem \ref{thm:NeighbourFormulaUsingOrthogonalModularForms} has an important consequence: When $n=26$, there exists (up to normalization) exactly one vector $v\in\C[I\!I_{24,0}^1]$ and $[\Lambda^+] - [\Lambda^-]\in\C[I\!I_{24,0}^{\det}]$ with eigenvalues as in Theorem \ref{thm:NeighbourFormulaUsingOrthogonalModularForms} for $f=\Delta$ \cite[Theorem 3.1]{NebeVenkov}. Consequently, the space $\mathcal{E}_{12}^1(\SGamma_{26},1)$ has dimension at most $2$. We will show in Corollary \ref{cor:SingularWeightFormsn=26} that the dimension is exactly $2$. Furthermore, for any $n$ the eigenvalue of $v$ as in Theorem \ref{thm:NeighbourFormulaUsingOrthogonalModularForms} suggests that it should have Arthur parameter $\pi_f[n/2-1]$. Assuming (H1), again there exist (up to normalization) exactly one vector $v_f^1\in\C[I\!I_{n-2,0}^1]$ and one vector $v_f^{\det}\in\C[I\!I_{n-2,0}^{\det}]$ with this Arthur parameter. Any holomorphic modular form of singular weight must take a linear combination of these values at the cusps. In Theorem \ref{thm:SingularWeightEisensteinSeries} we make this precise and describe the space of singular weight modular forms completely.

\subsection*{Translation from automorphic forms to classical modular forms}

As is well-known, holomorphic modular forms can be viewed as a special case of automorphic forms. For the convenience of the reader we make explicit the translation from the adelic viewpoint to the classical one.

\begin{lem}\label{lem:StrongapproximationG}
	There is a decomposition
	\begin{align}
		G(\A) &= G(\Q)\Sort^+(V(\R))K_\fin\label{eq:Decomposition1}
	\end{align}
\end{lem}
\begin{proof}
Let $g\in G(\A)$. Write $\nu_p\in\{0,1\}$ for the valuation of $\spin(g_p)$ modulo $2$ and choose
\[ a=\pm\prod_{p<\infty}p^{\nu_p}\in\Q^\times. \]
Only finitely many $\nu_p$ are non-zero. Multiplication on the left by $t(a^{-1},1)$ makes every finite spinor norm a unit square class. Choose the sign of $a$ so that the real component lies in $\Sort^+(V(\R))$. The spinor norm of $K_p$ contains every unit square class, as is seen from $t(u,1)$ for $u\in\Z_p^\times$. Multiplying on the right by suitable elements of $K_p$, we may therefore assume that all local spinor norms are trivial. The resulting element lifts to $\Spin(V)(\A)$. Strong approximation for the simply connected isotropic group $\Spin(V)$, away from infinity, now gives
\[ g\in G(\Q)\Sort^+(V(\R))K_\fin. \]
See \cite[Section 2.5]{GetzHahn} for strong approximation and its application to adelic decompositions.
\end{proof}

Now let $F\in\mathrm{M}_k(\SGamma_n,1)$. We define a function $\phi_F$ on $G(\A)$ by 
\[ \phi_F(g) = F|_k[g_\infty^+](Z_0) \]
where $g = \gamma g_\infty^+\kappa$ in decomposition \eqref{eq:Decomposition1}. This function is well-defined since
\[ G(\Q)\cap\Sort^+(V(\R))\prod_pK_p = \SGamma_n. \]
Let $\kappa_\infty=\kappa_\infty(\theta)\in K_\infty$ have $\Sort_2$ component
\[ \matr{\cos(\theta)}{\sin(\theta)}{-\sin(\theta)}{\cos(\theta)}. \]
Then
\begin{align*}
	\kappa_\infty(z+z'+&i(\rho+\rho')) \\
	&= (\cos(\theta)+i\sin(\theta))(z+z') + (-\sin(\theta) + i\cos(\theta))(\rho+\rho') \\
	&= e^{i\theta}(z+z' + i(\rho+\rho'))
\end{align*}
so that $j(\kappa_\infty,Z_0) = e^{i\theta}$. On the other hand let $\phi\in\mathcal{A}(G)$ such that for $\kappa_\infty(\theta)\in K_\infty$ we have $\phi(g\kappa_\infty) = e^{-ik\theta}\phi(g)$ for all $g\in G(\A)$. We define
\begin{align*}
	F_\phi:\mathcal{H}&\to\C \\
	Z&\mapsto j(g_Z,Z_0)^k\phi(g_Z)
\end{align*}
where $g_Z\in\Sort^+(V(\R))$ is any element such that $g_Z(Z_0) = Z$. Then we obtain

\begin{thm}\label{thm:ModularFormsAreAutomorphicForms}
	The assignments $F\mapsto\phi_F$ and $\phi\mapsto F_\phi$ are inverse to each other and define bijections between $\mathrm{M}_k(\SGamma_n,1)$ and the subspace of automorphic forms $\phi\in\mathcal{A}(G)$ with the following properties:
	\begin{enumerate}[(i)]
		\item $\phi$ has level $K_\fin$,
		\item $\phi(g\kappa_\infty) = e^{-ik\theta}\phi(g)$ for $\kappa_\infty(\theta)\in K_\infty$ and $g\in G(\A)$,
		\item for any $X\in\mfp_-$ the operator $R(X)$ annihilates $\phi$.
	\end{enumerate}
	We denote this space by $\holautform{G}{K_\fin}$.
\end{thm}
\begin{proof}
	The fact that the maps are inverse to each other is clear from the construction, provided that they are well-defined. So let $F\in\mathrm{M}_k(\SGamma_n,1)$ and $\phi\in\holautform{G}{K_\fin}$ be arbitrary. The left $G(\Q)$-invariance and right $K_\fin$-invariance of $\phi_F$ follows from construction and property (ii) follows from $j(\kappa_\infty(\theta),Z_0) = e^{i\theta}$. On the other hand, $F_\phi$ is well-defined because $K_\infty$ is the stabilizer of $Z_0$ and $j(\kappa_\infty(\theta),Z_0) = e^{i\theta}$. For any $\gamma\in\SGamma_n$ we have
	\begin{align*}
		F_\phi|_k[\gamma](Z) &= \phi(\gamma_\infty g_Z)j(\gamma,Z)^{-k}j(\gamma g_Z,Z_0)^k \\
		&= \phi\left(\gamma g_Z\prod_p\gamma_p^{-1}\right)j(g_Z,Z_0)^k \\
		&= \phi(g_Z)j(g_Z,Z_0)^k \\
		&= F_\phi(Z)
	\end{align*}
	by the left $G(\Q)$-invariance and the right $K_\fin$-invariance of $\phi$. It remains to show that $\phi_F$ has moderate growth and is $Z(\mfg)$-finite and that holomorphicity of $F$ is equivalent to property (iii) for $\phi_F$. 
	
	Under the standard identification of smooth functions satisfying
	(ii) with smooth sections of the weight-$k$ homogeneous line
	bundle on $\widetilde{\mathcal{H}}$, the Dolbeault operator is represented
	by the operators $R(X)$, $X\in\mfp_-$. Hence,
	\[ \overline{\partial} F=0
	\quad\Longleftrightarrow\quad
	R(X)\phi_F=0\qquad\text{for every }X\in\mfp_-.
	\]
	Moreover, holomorphicity of $F$ at the cusps, which follows automatically by the Koecher principle, gives polynomial bounds
	for $\phi_F$ on Siegel sets, by the usual Fourier-expansion
	estimates. This implies moderate growth.
	The argument is analogous to that for holomorphic Siegel
	modular forms. 
	
	We have seen that
	$R(\mfp_-)\phi_F=0$. Together with (ii), this implies that
	the $U(\mfg)$-module generated by $\phi_F$ is a quotient of
	the generalized Verma module
	\[ U(\mfg)\otimes_{U(\mfk\oplus\mfp_-)}
	\C_{\tau_k},
	\]
	where $\tau_k(\kappa_\infty(\theta))=e^{-ik\theta}$ and	$\mfp_-$ acts trivially on $\C_{\tau_k}$. This generalized Verma module has an infinitesimal character, so that $\phi_F$ is a simultaneous eigenfunction for $Z(\mfg)$.
\end{proof}

We need a decomposition similar to \eqref{eq:Decomposition1}, but in terms of $P$. For each $[D]\in\widetilde{I\!I}_{n-2,0}$ choose an orientation-preserving rational isometry $\iota_D:D\otimes\Q\to W$. Extend it to $\gamma_D\in G(\Q)$ by
\[ \gamma_D(z)=z,\quad\gamma_D(z')=z',\quad
\gamma_D(\rho_D)=\rho,\quad\gamma_D(\rho_D')=\rho'. \]
Such an extension exists by the Witt decomposition. The choices of the positive cones are made so that $\gamma_D\in\Sort^+(V(\R))$.

\begin{prp}\label{prp:DecompositionStrongApproximationWithParabolic}
The elements $\gamma_D$ represent the double quotient
\[ P(\Q)\backslash G(\Q)/\SGamma_n. \]
Consequently,
\begin{equation*}\label{eq:Decomposition2}
G(\A)=\coprod_{[D]\in\widetilde{I\!I}_{n-2,0}}P(\Q)\gamma_D\Sort^+(V(\R))K_\fin.
\end{equation*}
For every $D$ one can choose $\omega_{D,\fin}\in\Sort_W(\A_\fin)$ such that
\[ \omega_{D,\fin}(D_0\otimes\widehat{\Z})=\iota_D(D\otimes\widehat{\Z}),\qquad
\gamma_{D,\fin}=(1,\omega_{D,\fin})\kappa_D\quad\text{with }\kappa_D\in K_\fin. \]
If $g_\infty\in\Sort^+(V(\R))$ and
$\gamma_Dg_\infty=(m_\infty,\omega_\infty)n_\infty\kappa$ is an Iwasawa decomposition, then
\begin{align*}
(-\rho_D,g_\infty Z_0)&=m_\infty i,\\
j(g_\infty,Z_0)&=\det(m_\infty)^{-1/2}j(m_\infty,i)j(\kappa,Z_0).
\end{align*}
\end{prp}
\begin{proof}
The first assertion is the classification of the $1$-dimensional cusps by the proper classes of $F_D^\perp/F_D$. The adelic decomposition follows from Lemma \ref{lem:StrongapproximationG}, its disjointness follows by intersecting with $G(\Q)$. Locally $D$ and $D_0$ are isometric and their local integral orthogonal groups contain elements of determinant $-1$. We may therefore choose the local isometries to have determinant $1$, giving $\omega_{D,\fin}$. Now $\kappa_D\define(1,\omega_{D,\fin})^{-1}\gamma_D$ acts as the identity on $\langle z,z'\rangle$, maps $\rho_D$ to $\rho$, $\rho_D'$ to $\rho'$ and $D\otimes\widehat{\Z}$ to $D_0\otimes\widehat{\Z}$, so it lies in $K_\fin$. Finally, $j(\gamma_D,Z)=1$, the $N$-action fixes the coordinate $\tau$ and the two real identities follow from Lemma \ref{lem:ActionOfM} and the cocycle identity for $j$.
\end{proof}

For $\det(m)>0$, the normalized elliptic slash operator is
\[ f|_k[m](\tau)=j(m,\tau)^{-k}f(m\tau). \]
Let $f=\sum_D f_D\otimes[D]\in\mathrm{M}_k(\SL_2(\Z))\otimes\C[\widetilde{I\!I}_{n-2,0}]$. Define
\[ \varphi_f((m,\omega))=f_D|_k[m_\infty^+](i), \]
where $m=\mu m_\infty^+h\in\GL_2(\Q)\GL_2^+(\R)\GL_2(\widehat{\Z})$ and $[D]=[\omega_\fin(D_0\otimes\widehat{\Z})\cap W]$. This uses the normalized elliptic slash operator defined above, in particular, $\varphi_f$ has trivial central character on the $\GL_2$ factor. Conversely, choose $m_\tau\in\SL_2(\R)$ with $m_\tau i=\tau$ and put
\[ f_\varphi(\tau)=\sum_D j(m_\tau,i)^k\varphi((m_\tau,\omega_{D,\fin}))\otimes[D]. \]

For the complexified Lie algebra $\mathfrak m$ of $M(\R)$ we have
\[ \mathfrak m=\mathfrak a_M\oplus\mathfrak k_M\oplus\mathfrak p_{M,+}\oplus\mathfrak p_{M,-}, \]
where $\mathfrak a_M$ is the split central direction and $\mathfrak p_{M,\pm}$ are the holomorphic and antiholomorphic directions for $\GL_2^+(\R)/(\R_{>0}\Sort_2(\R))\cong\mathbb H$.

\begin{thm}\label{thm:ModularFormsAreAutomorphicFormsCusps}
The assignments $f\mapsto\varphi_f$ and $\varphi\mapsto f_\varphi$ define inverse bijections between $\mathrm{M}_k(\SL_2(\Z))\otimes\C[\widetilde{I\!I}_{n-2,0}]$ and the subspace of $\mathcal{A}(M)$ satisfying
\begin{enumerate}[(i)]
\item $\varphi$ has level $K_{M,\fin}$ and trivial central character on $\GL_2(\A)$
\item $\varphi$ is invariant under $\Sort_W(\R)$ and has weight $-k$ under the $\Sort_2(\R)$ factor of $M(\R)\cap K_\infty$
\item $R(X)\varphi=0$ for all $X\in\mathfrak p_{M,-}$.
\end{enumerate}
Denote this space by $\holautform{M}{K_{M,\fin}}$. The diagram
\[ \begin{tikzcd}
\mathrm{M}_k(\SGamma_n,1) \arrow[d,"\Phi"] \arrow[rr,"\cong"] && \holautform{G}{K_\fin} \arrow[d,"\psi^{-k/2}(\cdot)_P"] \\
\mathrm{M}_k(\SL_2(\Z))\otimes\C[\widetilde{I\!I}_{n-2,0}] \arrow[rr,"\cong"] && \holautform{M}{K_{M,\fin}}
\end{tikzcd} \]
commutes.
\end{thm}
\begin{proof}
The first assertion is the usual adelization of elliptic modular forms, together with the finite double-coset description of $\Sort_W$ (see \cite[Chapter 6]{GetzHahn}). Holomorphicity and moderate growth exclude negative Fourier powers at the elliptic cusp.

For the constant term, fix $D$ and write $N_D=\gamma_D^{-1}N\gamma_D$. Integrating $F(nZ)$ first over the central coordinate $\Z\backslash\R$ extracts the zeroth Fourier--Jacobi coefficient of $F$, which as we have seen is independent of $\mathfrak z$. Thus, with volume-one measures,
\[ \int_{N_D(\Z)\backslash N_D(\R)}F(nZ)\mathrm{d}n=\Phi_D(F)(\tau),\qquad \tau=(-\rho_D,Z). \]
The same identity follows directly by differentiating the average in $\Ima(\tau')$ and using holomorphicity and periodicity, then taking $\Ima(\tau')\to\infty$.

Let $g=(m,\omega)\in M(\A)$. Using left $M(\Q)$-invariance and right $K_{M,\fin}$-invariance, we may take $m=(m_\infty,1)$ with $\det(m_\infty)>0$ and $\omega_\fin=\omega_{D,\fin}$. Put $g_\infty^+=\gamma_D^{-1}(m_\infty,\omega_\infty)$. Then $g=\gamma_Dg_\infty^+\kappa_D^{-1}$. Conjugating the unipotent variable by $\gamma_D$ before removing the rational factor gives
\begin{align*}
(\phi_F)_P(g)
&=\int_{N_D(\Q)\backslash N_D(\A)}\phi_F(ng_\infty^+)\mathrm{d}n\\
&=\int_{N_D(\Z)\backslash N_D(\R)}\phi_F(n_\infty g_\infty^+)\mathrm{d}n_\infty\\
&=j(g_\infty^+,Z_0)^{-k}\Phi_D(F)(m_\infty i).
\end{align*}
The second equality uses strong approximation for $N_D$ and right $K_\fin$-invariance. Here $j(n_\infty,Z)=1$ and $j(g_\infty^+,Z_0)=\det(m_\infty)^{-1/2}j(m_\infty,i)$. Multiplying by $\psi(g)^{-k/2}=\det(m_\infty)^{-k/2}$ yields exactly $\Phi_D(F)|_k[m_\infty](i)$.
\end{proof}

\subsection*{Hecke operators}

We now define Hecke operators acting on the space of holomorphic modular forms. They coincide with the previously defined Hecke operators acting on automorphic forms.

Let $\Delta_n = \Ort^+(I\!I_{n,2}\otimes_\Z\Q)$ and $\SDelta_n = \Sort^+(I\!I_{n,2}\otimes_\Z\Q)$ and $(\Gamma,\Delta)$ will be either $(\Gamma_n,\Delta_n)$ or $(\SGamma_n,\SDelta_n)$. Then $(\Gamma,\Delta)$ is a \textit{Hecke pair}, i.e.\ for all $\delta\in \Delta$ the double coset $\Gamma \delta\Gamma$ is the union of finitely many right cosets
\[ \Gamma \delta\Gamma  = \bigcup_{i=1}^h\Gamma \delta_i. \]
Now let $\mathscr{L}(\Gamma ,\Delta)$ and $\mathscr{H}(\Gamma ,\Delta)$ be the $\C$-vector spaces generated by the elements of $\Gamma \backslash \Delta$ and $\Gamma \backslash \Delta/\Gamma $ respectively. An element $\delta\in \Delta$ acts on $\mathscr{L}(\Gamma ,\Delta)$ by right multiplication. We denote by
\[ \mathscr{L}(\Gamma ,\Delta)^\Gamma \define\{T\in\mathscr{L}(\Gamma ,\Delta)\mid T\gamma=T\text{ for }\gamma\in \Gamma \} \]
the submodule of $\Gamma $-invariants. Then the map given by
\[ \Gamma \delta\Gamma \mapsto \sum_{i=1}^h\Gamma \delta_i \]
defines an isomorphism $\mathscr{H}(\Gamma ,\Delta)\xrightarrow{\sim}\mathscr{L}(\Gamma ,\Delta)^\Gamma $ and we will identify these two spaces. Let $\delta,\epsilon\in \Delta$ and 
\[ \Gamma \delta\Gamma  = \bigcup_i \Gamma \delta_i,\quad \Gamma \epsilon\Gamma  = \bigcup_j \Gamma \epsilon_j. \]
We define the product of $\Gamma \delta\Gamma $ and $\Gamma \epsilon\Gamma $ to be 
\[ (\Gamma \delta\Gamma )\cdot(\Gamma \epsilon\Gamma ) = \sum_{i,j}\Gamma \delta_i\epsilon_j\in\mathscr{L}(\Gamma ,\Delta). \]
Clearly, $(\Gamma \delta\Gamma )\cdot(\Gamma \epsilon\Gamma )\in\mathscr{L}(\Gamma ,\Delta)^\Gamma $ and so this extends to a well-defined product on $\mathscr{H}(\Gamma ,\Delta)$.

Choose $r\in\Gamma_n$ with $\det(r)=-1$. Conjugation by $r$ induces an involution $\sigma$ of $\mathscr{H}(\SGamma_n,\SDelta_n)$. Write
\[ \mathscr{H}(\SGamma_n,\SDelta_n)=\mathscr{H}_0\oplus\mathscr{H}_1,\qquad
\mathscr{H}_i=\{T:\sigma(T)=(-1)^iT\}. \]
This is a $\Z/2\Z$-grading: $\mathscr{H}_i\mathscr{H}_j\subset\mathscr{H}_{i+j}$. The even operators preserve the character spaces, whereas
\[ T:\mathrm{M}_k(\Gamma_n,\chi)\longrightarrow\mathrm{M}_k(\Gamma_n,\det\chi),\qquad T\in\mathscr{H}_1. \]
Indeed, $r$ acts on these spaces by $F\mapsto F|_k[r]$ and $rTr^{-1}=(-1)^iT$. Furthermore, $\mathscr{H}(\Gamma_n ,\Delta_n)$ embeds into $\mathscr{H}(\SGamma_n ,\SDelta_n)$ via $T \mapsto T\cap\SDelta_n$ and its image is $\mathscr{H}_0$.

We use the following identifications between classical and adelic Hecke actions.

\begin{prp}\label{prp:HeckeOperatorsAdelicClassicalAreTheSame}
	To a double coset $T = \SGamma_n\delta\SGamma_n$ we associate a function $f\in\mathcal{H}(G)$ by 
	\[ f(g) = \begin{cases}
		1 & \text{if } g\in G(\widehat{\Z})\delta^{-1}G(\widehat{\Z}) \\
		0 & \text{else}.
	\end{cases} \]
	Then $\mathscr{H}(\SGamma_n ,\SDelta_n)\cong\mathcal{H}(G)$ and the diagram 
	\[ \begin{tikzcd}
		\mathrm{M}_k(\SGamma_n,1) \arrow[d,"\cong"] \arrow[rr, "T"]                                                  &  & \mathrm{M}_k(\SGamma_n,1) \arrow[d, "\cong"]                           \\
		\holautform{G}{K_\fin} \arrow[rr, "f"] &  & \holautform{G}{K_\fin}
	\end{tikzcd} \]
	commutes.
\end{prp}

\begin{prp}\label{prp:HeckeOperatorsAdelicClassicalAreTheSameCusps}
	The assignment
\[ \SL_2(\Z)\delta\SL_2(\Z)\longmapsto
\mathbf1_{\GL_2(\widehat{\Z})\delta^{-1}\GL_2(\widehat{\Z})} \]
defines an isomorphism
\[ \mathscr{H}(\SL_2(\Z),\GL_2^+(\Q))\otimes\mathcal{H}(\Sort_W)\xrightarrow{\sim}\mathcal{H}(M). \]
Under this isomorphism the adelization in Theorem \ref{thm:ModularFormsAreAutomorphicFormsCusps} is Hecke-equivariant. On forms with trivial central character the action factors through
\[ \mathscr{H}(\PGL_2^+(\Z),\PGL_2^+(\Q))\otimes\mathcal{H}(\Sort_W), \]
obtained by imposing that every scalar double coset acts as the identity. Here $\PGL_2^+(\Z)$ denotes the image of $\SL_2(\Z)$.

We denote by $T(p)$ the operator in $\mathscr{H}(\PGL_2^+(\Z),\PGL_2^+(\Q))$ which is the image of
\[ T(p)=p^{k/2-1}\sum_i[\delta_i],\qquad
\SL_2(\Z)\diag(p,1)\SL_2(\Z)=\coprod_i\SL_2(\Z)\delta_i. \]
\end{prp}

As in the adelic setting the Hecke algebras are compatible with the $\Phi$-operator.

\begin{prp}\label{prp:HeckeCommutationRelation}
	Let $p$ be a prime. There exists a homomorphism
	\[ \iota:\mathscr{H}(\SGamma_n ,\SDelta_n)\to\mathscr{H}(\PGL_2^+(\Z),\PGL_2^+(\Q))\otimes\mathcal{H}(\Sort_W) \]
	such that the diagram
	\[ \begin{tikzcd}
		\mathrm{M}_k(\SGamma_n,1) \arrow[rr, "T"] \arrow[d, "\Phi"] &  & \mathrm{M}_k(\SGamma_n,1) \arrow[d, "\Phi"] \\
		\mathrm{M}_k(\SL_2(\Z))\otimes_\C\C[\widetilde{I\!I}_{n-2,0}] \arrow[rr, "\iota(T)"] &  & \mathrm{M}_k(\SL_2(\Z))\otimes_\C\C[\widetilde{I\!I}_{n-2,0}]   
	\end{tikzcd} \]
	commutes. 
\end{prp}
\begin{proof}
	The diagram fits into the larger commutative diagram
		\[ \begin{tikzcd}
		\mathrm{M}_k(\SGamma_n,1) \arrow[rr, "T"] \arrow[d, "\cong"] \arrow[ddd, "\Phi", bend right=49, shift right=14] &  & \mathrm{M}_k(\SGamma_n,1) \arrow[d, "\cong"] \arrow[ddd, "\Phi", bend left=49,shift left=14] \\
		\holautform{G}{K_\fin} \arrow[rr, "f"] \arrow[d, "\psi^{-k/2}\cdot(\cdot)_P"] &  & \holautform{G}{K_\fin} \arrow[d, "\psi^{-k/2}\cdot(\cdot)_P"] \\
		\holautform{M}{K_{M,\fin}} \arrow[rr, "\psi^{-(n-k-1)/2}\cdot\iota_M^G(f)"] \arrow[d, "\cong"]  &  &  \holautform{M}{K_{M,\fin}} \arrow[d, "\cong"] \\
		\mathrm{M}_k(\SL_2(\Z))\otimes_\C\C[\widetilde{I\!I}_{n-2,0}] \arrow[rr, "\iota(T)"] &  & \mathrm{M}_k(\SL_2(\Z))\otimes_\C\C[\widetilde{I\!I}_{n-2,0}]   
	\end{tikzcd} \]
	which is immediate by \eqref{eq:HeckeCommutationRelationsAdelic}, Theorem \ref{thm:ModularFormsAreAutomorphicFormsCusps} and Propositions \ref{prp:HeckeOperatorsAdelicClassicalAreTheSame} and \ref{prp:HeckeOperatorsAdelicClassicalAreTheSameCusps}.
\end{proof}

The simplest operator in $\mathscr{H}(\Gamma_n ,\Delta_n)$ is the double coset operator
\[ T_1(p^2) \define \Gamma_nt(p,1)\Gamma_n \]
for a prime $p$. We use the notation $T_1(p^2)$ for this operator, as in \cite{thesis_mmueller}. In \cite{KriegRoemerSchaps} for this operator an explicit decomposition into right cosets was determined and its action on the Eisenstein series $\Psi_{I\!I_{n,2}}^{\mathrm{add}}(E_k)$ was computed.
\begin{prp}\label{prp:T_1(p^2)}
	Let $p$ be a prime then
	\[ \iota(T_1(p^2)) = (p^{n-k}+p)T(p)\otimes\id + p^2\cdot\id\otimes T_{\Z/p\Z}.  \]
\end{prp}
\begin{proof}
Put $H_p=\mathbf1_{\GL_2(\Z_p)\diag(p,1)\GL_2(\Z_p)}$ and $Z_p=\mathbf1_{p\GL_2(\Z_p)}$. Before passing to the quotient by the center, the normalized Satake formulas are
\begin{align*}
\Sat(T_1(p^2))&=p^{n/2}[\Std_{\Sort_{n+2}}],\\
\Sat(T_{\Z/p\Z})&=p^{n/2-2}[\Std_{\Sort_{n-2}}],\\
\Sat(H_p)&=p^{1/2}[\Std_{\GL_2}],\qquad \Sat(Z_p)=[\det].
\end{align*}
See \cite[Section 6.2]{ChenevierLannes}. Since
\[ \Std_{\Sort_{n+2}}|_{\widehat M}=\Std_{\GL_2}\oplus\Std_{\Sort_{n-2}}\oplus\Std_{\GL_2}^\vee,\qquad
\Std_{\GL_2}^\vee={\det}^{-1}\otimes\Std_{\GL_2}, \]
we obtain
\[ \iota_M^G(T_1(p^2))=p^{(n-1)/2}(H_p+Z_p^{-1}H_p)\otimes1+p^2\cdot1\otimes T_{\Z/p\Z}. \]
Multiplication by $\psi_p^{-(n-k-1)/2}$ multiplies the two $\GL_2$ terms by $p^{(n-k-1)/2}$ and $p^{-(n-k-1)/2}$, respectively. Now $Z_p$ acts trivially and $H_p$ acts as $p^{1-k/2}T(p)$. The coefficients are therefore $p^{n-k}$ and $p$, as asserted. In particular, the restriction of representations is taken in $\widehat M=\GL_2(\C)\times\Sort_{n-2}(\C)$ before passing to the trivial-central-character quotient.
\end{proof}

Consider the irreducible representation $\bigwedge^{n/2+1}\Std_{\Ort_{n+2}}$ of $\Ort_{n+2}$. As a representation of $\Sort_{n+2}$ it decomposes into two irreducible representations 
\[ \bigwedge^{n/2+1}\Std_{\Ort_{n+2}}|_{\Sort_{n+2}} = V_+ \oplus V_-. \]
We define the operator $T_-(p^2)\in\mathscr{H}(\SGamma_n ,\SDelta_n)$ as the operator corresponding to
\[ \Sat^{-1}([V_+] - [V_-]). \]
Then $T_-(p^2)\in\mathscr{H}_1$. Define $T_{W,-}(p^2)$ in $\mathcal{H}(\Sort_W)$ by the analogous virtual representation for $\Sort_{n-2}$.
\begin{prp}\label{prp:OddHeckePhi}
	Let $p$ be a prime then
	\[ \iota(T_-(p^2))=
	\bigl((p^{n-k-1}+p^{-(n-k-1)}+2)\id-p^{-(k-1)}T(p)^2\bigr)\otimes T_{W,-}(p^2).  \]
\end{prp}
\begin{proof}
	As for $T_1(p^2)$ we compute the restriction of $[V_+]-[V_-]$ to $\widehat{M}$. This can be done by considering characters. Put $r=n/2+1$. Choose the labels $V_+,V_-$ such that on a diagonal torus of $\Sort_{2r}(\C)$
	\[ \tr_{V_+-V_-}(\diag(x_1,\ldots,x_r,x_1^{-1},\ldots,x_r^{-1}))=\prod_{i=1}^r(x_i-x_i^{-1}). \]
	This follows by taking the difference of the characters of the two summands of the exterior power. By the same argument $[V_{W,+}]-[V_{W,-}]$ has character $\prod_{i=3}^r(x_i-x_i^{-1})$. On the $\mathrm{GL}_2$-factor we have
	\begin{align*}
		(x_1-x_1^{-1})(x_2-x_2^{-1})
		&=x_1x_2+(x_1x_2)^{-1}-x_1/x_2-x_2/x_1\\
		&=x_1x_2+(x_1x_2)^{-1}+2-(x_1x_2)^{-1}(x_1+x_2)^2.
	\end{align*}
	Since $x_1x_2$ is the character of ${\det}_{\GL_2}$ and $x_1+x_2$ is the character of $\Std_{\GL_2}$, restriction to $\widehat M$ gives
	\[ \operatorname{Res}_{\widehat M}([V_+]-[V_-])=
		\bigl([\det]+[\det]^{-1}+2-[\det]^{-1}[\Std_{\GL_2}]^2\bigr)\otimes([V_{W,+}]-[V_{W,-}]). \]
	Applying the twist in the Hecke constant term and then passing to trivial central character, we obtain the claimed formula.
\end{proof}
Here $T(p)^2$ denotes the square of the elliptic Hecke operator, not $T(p^2)$.

We consider again $\Phi_{12}\in\mathcal{E}_{12}^1(\Gamma_{26},\det)$. Recall that
\[ \Phi(\Phi_{12}) = \Delta\otimes([\Lambda^+] - [\Lambda^-]) \]
and
\[ T_{\Z/p\Z} ([\Lambda^+] - [\Lambda^-]) = \tau(p)\frac{p^{12}-1}{p-1}([\Lambda^+] - [\Lambda^-]). \]
Since $\mathcal{E}_{n/2-1}^0(\Gamma_n,\det) = \{0\} = \mathrm{S}_{n/2-1}(\Gamma_n,\det)$ and $I\!I_{24,0}$ contains only one chiral lattice, $\dim(\mathcal{E}_{n/2-1}^1(\Gamma_{26},\det)) = 1$ and $\Phi_{12}$ must be a simultaneous eigenform for $\mathscr{H}(\Gamma_n ,\Delta_n)$. Using Proposition \ref{prp:T_1(p^2)}, we deduce that the eigenvalue of $\Phi_{12}$ for $T_1(p^2)$ is
\begin{align*}
	(p^{26-12} + p)\cdot\tau(p) &+ p^2\cdot\tau(p)\frac{p^{12}-1}{p-1} \\
	&=p\cdot \tau(p)\frac{(p^{13}+1)(p-1) + p^{13}-p}{p-1} \\
	&= p\cdot \tau(p)\frac{p^{14}-1}{p-1}.
\end{align*}

Denote by $\vartheta_{I\!I_{m,0}}^{(g)}:\C[I\!I_{m,0}]\to\mathrm{M}_{m/2}(\Sp_{2g}(\Z))$ the map given by
\[ \vartheta_{I\!I_{m,0}}^{(g)}([D]) = \left(Z\mapsto\frac{1}{|\Aut(D)|}\sum_{\lambda\in D^g}e^{\pi i\tr((\lambda,\lambda)Z)},\qquad Z\in\mathbb{H}_g\right). \]

Then we have
\begin{prp}\label{prp:Psi12}
	There exists a non-zero form $\Psi_{12}\in\mathrm{M}_{12}(\Gamma_{26},1)$. We have $\Phi(\Psi_{12}) = \Delta\otimes v$ where $v\in\C[I\!I_{24,0}]$ is the unique vector such that $\vartheta_{I\!I_{24,0}}^{(12)}(v)$ is the Ikeda-lift of $\Delta$.
\end{prp}
\begin{proof}
	Write $v_\Delta^{\det}=[\Lambda^+]-[\Lambda^-]$. Its standard parameter is $\pi_\Delta[12]$ by \cite[Section 7.5]{ChenevierLannes}. Let $\alpha_p,\alpha_p^{-1}$ be the unitary Satake eigenvalues of $\pi_\Delta$, thus $\alpha_p+\alpha_p^{-1}=p^{-11/2}\tau(p)$ and $|\alpha_p|=1$ by \cite{Deligne}. The standard eigenvalues for $\pi_\Delta[12]$ are
	\[ \alpha_p^{\pm1}p^{j-11/2},\qquad 0\leq j\leq11. \]
	None is $1$ or $-1$. Hence, the square of the character $\prod_i(x_i-x_i^{-1})$ is non-zero on this parameter. This square is invariant under the outer involution and therefore acts on $v_\Delta^{\det}$ by a non-zero scalar. It follows that $T_{W,-}(p^2)v_\Delta^{\det}\ne0$. By Proposition \ref{prp:OddHeckePhi},
	\[ \Phi(T_-(p^2)\Phi_{12})=
	\bigl(p^{13}+p^{-13}-(\alpha_p^2+\alpha_p^{-2})\bigr)\Delta\otimes T_{W,-}(p^2)v_\Delta^{\det}. \]
	The scalar is positive since $\alpha_p^2+\alpha_p^{-2}\leq2<p^{13}+p^{-13}$. Thus $T_-(p^2)\Phi_{12}\ne0$ for every prime $p$ and it has trivial character. Since the Hecke algebra is commutative, the non-zero vector $T_{W,-}(p^2)v_\Delta^{\det}\in\C[I\!I_{24,0}]$ has the same eigenvalues as $v_\Delta^{\det}$ for every operator in $\mathcal{H}(\Ort_W)$.

	This common eigenspace is one-dimensional: already $T_{\Z/2\Z}$ has distinct eigenvalues on $\C[I\!I_{24,0}]$ and the eigenvalue in question is $\tau(2)(2^{12}-1)=-98280$ (see \cite[Theorem 3.1]{NebeVenkov}). Its degree-$12$ theta lift is non-zero and proportional to the Ikeda lift of $\Delta$ by \cite[Theorem 7.3.2]{ChenevierLannes}. Normalize a scalar multiple $\Psi_{12}$ of $T_-(p^2)\Phi_{12}$ so that its boundary vector has theta lift equal to the chosen normalization of the Ikeda lift. The injectivity of the degree-$12$ theta map in rank $24$ (see \cite{BorcherdsFreitagWeissauer}) gives the asserted uniqueness of $v$.
\end{proof}

\begin{cor}\label{cor:SingularWeightFormsn=26}
	We have
	\[ \mathrm{M}_{12}(\SGamma_{26},1) = \underbrace{\C \Psi_{I\!I_{26,2}}^{\mathrm{add}}(1)}_{\subset\mathcal{E}_{12}^0(\Gamma_{26},1)} \oplus \underbrace{\C \Psi_{12}}_{\subset\mathcal{E}_{12}^1(\Gamma_{26},1)} \oplus \underbrace{\C \Phi_{12}}_{\subset\mathcal{E}_{12}^1(\Gamma_{26},\det)}. \]
\end{cor}
\begin{proof}
The two forms constructed above and $\Psi_{I\!I_{26,2}}^{\mathrm{add}}(1)$ have independent boundary values. Hence, they are linearly independent. Furthermore, singular-weight cusp forms vanish and by Theorem \ref{thm:NeighbourFormulaUsingOrthogonalModularForms} and the arguments thereafter no other singular weight form can exist.
\end{proof}

\subsection*{Eisenstein series}\label{sec:HolomorphicEisensteinSeries}

As before let $\Gamma$ be either $\Gamma_n$ or $\SGamma_n$. Let $P_D\define\Stab_{\Ort(V)}(F_D)$ and define $\Gamma_D \define P_D(\Z)\cap \Gamma$. Let $f\in\mathrm{S}_k(\SL_2(\Z))$ and $[D]\in I\!I_{n-2,0}$. If $\Gamma = \Gamma_n$ and $\chi = \det$, then $D$ must be chiral. We define the \emph{Klingen--Eisenstein series}
\begin{align*}
	E_{\Gamma,f,[D],\chi}(Z,s) &\define \sum_{\gamma\in\Gamma_D\backslash\Gamma}\chi(\gamma)f((-\rho_D,\cdot))|_k[\gamma](Z)\frac{(-\q(\Ima(\gamma Z)))^{s}}{(-\rho_D,\Ima(\gamma Z))^{s}} \\
	&= \sum_{\gamma\in\Gamma_D\backslash\Gamma}\chi(\gamma)f((-\rho_D,\cdot))|_k[\gamma](Z)\!\left(\frac{-\q(\Ima(Z))}{|j(\gamma,Z)|^2(-\rho_D,\Ima(\gamma Z))}\right)^{s}.
\end{align*}
For $\SGamma_n$ we use proper lattice classes. For $\Gamma_n$ and chiral $D$, the sum is the sum or difference of the two series for $\SGamma_n$ of proper lattice classes, according to $\chi$, equivalently it is twice the corresponding character projection of either series, with compatible choices of orientation. For achiral $D$ only the trivial-character boundary value is possible. Now let $f = \sum_{[D]\in \widetilde{I\!I}_{n-2,0}}f_{[D]}[D]\in\mathrm{S}_k(\SL_2(\Z))\otimes_\C\C[\widetilde{I\!I}_{n-2,0}]$ and define
\[ E_{f}(Z,s) \define\sum_{[D]\in \widetilde{I\!I}_{n-2,0}}E_{\SGamma_n,f_{[D]},[D],1}(Z,s) \]
and $E_f(Z) \define E_f(Z,0)$. We then have

\begin{thm}\label{thm:EisensteinseriesClassicalIsAdelic}
Let $f\in\mathrm S_k(\SL_2(\Z))$ be a normalized eigenform and let $v\in\C[\widetilde{I\!I}_{n-2,0}]$ be a non-zero Hecke eigenvector. Define the standard section
\[ \varphi_s((m,\omega)n\kappa)=\psi((m,\omega))^{s+k/2}\varphi_{f\otimes v}((m,\omega))j(\kappa,Z_0)^{-k}, \]
which is right $K_\fin$-invariant. It lies in $I(\pi_f\otimes\pi_v,s-(n-k-1)/2)$. For $g_Z\in\Sort^+(V(\R))$ with $g_ZZ_0=Z$,
\[ E_{f\otimes v}(Z,s)=j(g_Z,Z_0)^kE(g_Z,\varphi_s). \]
It converges absolutely if $\Rea(s)>n-1-k/2$, extends meromorphically to $s\in\C$ and transforms with weight $k$. In particular, $E_{f\otimes v}(Z,0)$ is holomorphic if $k>2n-2$.

Now assume that $v$ has Arthur parameter $\psi=\bigoplus_{i=1}^m\pi_i[d_i]$. Then $E_{f\otimes v}(Z)\in\mathrm{M}_k(\SGamma_n,1)$ if and only if one of the following conditions hold
\begin{enumerate}[(i)]
	\item $k \geq n+2$
	\item $k=n$ and $L(\frac{1}{2},\pi_f\otimes\pi_i)=0$ for some $i$ with $d_i=1$
	\item $k = n-2$ and $L(\frac{1}{2},\pi_f\otimes\pi_i)\not=0$ whenever $d_i = 1$
	\item $k < n-2$, $\pi_i\cong\pi_f$ and $d_i=n-k-2$ for some $i$ and $L(\frac{1}{2},\pi_f\otimes\pi_i)\not=0$ whenever $d_i = n-k-1$.
\end{enumerate}
In this case 
\[ \Phi(E_{f\otimes v}) = f\otimes v. \]
\end{thm}
\begin{proof}
Proposition \ref{prp:DecompositionStrongApproximationWithParabolic} gives
\[ P(\Q)\backslash G(\Q)=\coprod_D\gamma_D(\Gamma_D\backslash\SGamma_n). \]
For $\gamma\in\SGamma_n$, use the real Iwasawa decomposition of $\gamma_D\gamma g_Z$ and the finite decomposition $\gamma_{D,\fin}=(1,\omega_{D,\fin})\kappa_D$. Then Proposition \ref{prp:DecompositionStrongApproximationWithParabolic} gives exactly
\[ \varphi_s(\gamma_D\gamma g_Z) = v_D f((-\rho_D,\gamma Z))j(\gamma g_Z,Z_0)^{-k}\psi(\gamma_D\gamma g_Z)^s. \]
By Proposition \ref{prp:DecompositionStrongApproximationWithParabolic} and Lemma \ref{lem:characterIsImaginaryPart},
\[ \psi(\gamma_D\gamma g_Z)^s=
\left(\frac{-\q(\Ima(Z))}{|j(\gamma,Z)|^2(-\rho_D,\Ima(\gamma Z))}\right)^s. \]
Multiplying by $j(g_Z,Z_0)^k$ and summing proves the identity. Convergence and continuation now follow from Theorem \ref{thm:EisensteinSeriesDefinition}.

Put $s_0=-(n-k-1)/2$, the normalized inducing parameter corresponding to the classical value $s=0$. At this point the standard section is annihilated by $\mathfrak p_-$. If $m(s_0)=0$, Proposition \ref{prp:HolomorphicityConstantTerms} proves regularity and holomorphicity and the constant term formula and Theorem \ref{thm:ModularFormsAreAutomorphicFormsCusps} give $\Phi(E_{f\otimes v})=f\otimes v$.

Conversely, if the Eisenstein value is regular and holomorphic, its constant term has central exponent $k/2$ before normalization. The two exponents in Langlands' formula are $k/2$ and $n-1-k/2$. They are distinct since $k$ and $n$ are even. Hence, the coefficient of the second exponent must vanish, i.e.\ $m(s_0)=0$. The stated criterion follows from Theorem \ref{thm:mainThmeoremIntertwinging}.
\end{proof}

We remark that in the above theorem one may also take an eigenvector $v\in\C[I\!I_{n-2,0}^\chi]$ for $\mathcal{H}(\Ort_W)$, provided its special orthogonal constituents have the same standard parameter. Then $E_{f\otimes v}(Z)\in\mathrm{M}_k(\Gamma_n,\chi)$.

The elements in $\mathcal{E}_k^1(\Gamma_n,\chi)$ are all given by Eisenstein series.

\begin{prp}\label{prp:EisensteinBasis}
	Choose a normalized eigenbasis $f_1,\ldots,f_h$ of $\mathrm S_k(\SL_2(\Z))$ and a simultaneous eigenbasis $v_1,\ldots,v_r$ of $\C[I\!I_{n-2,0}^\chi]$. Let $\mathcal{B}_k^\chi$ consist of the pairs $(f_i,v_j)$ for which $E_{f_i\otimes v_j}\in\mathrm{M}_k(\Gamma_n,\chi)$. Then
	\[ \bigl(E_{f\otimes v}\bigr)_{(f,v)\in\mathcal{B}_k^\chi} \]
	is a basis of $\mathcal{E}_k^1(\Gamma_n,\chi)$. Under (H1), the standard parameters give the characterization in Theorem \ref{thm:EisensteinseriesClassicalIsAdelic}. The analogous statement holds for $\SGamma_n$ with a simultaneous eigenbasis for $\mathcal{H}(\Sort_W)$.
\end{prp}
\begin{proof}
	Unfolding shows that every Eisenstein series in the statement is orthogonal to cusp forms. Their boundary values are the independent vectors $f_i\otimes v_j$, so they are linearly independent.
	
	For spanning, let $F\in\mathcal{E}_k^1(\Gamma_n,\chi)$ and write $\Phi(F)=\sum_{i,j}c_{ij}f_i\otimes v_j$. Set $s_0 = -(n-k-1)/2$. According to \cite[Theorem 1.4]{FrankeSchwermer} and \cite[Section III.2.6]{MoeglinWaldspurger}, there exist $I(\pi_{f_i}\otimes\pi_{v_j},s)$-valued polynomials $Q_{ij}\in\C[s]\otimes_\C I(\pi_{f_i}\otimes\pi_{v_j},s)$ and integers $N_{ij}\in\Z_{\geq0}$ such that
	\[ \phi_F(g) = \sum_{i,j}[(s-s_0)^{N_{ij}}]_{E(g,Q_{ij}(s))} \]
	where $[(s-s_0)^{N}]_{A(s)}$ denotes the $N$-th Laurent coefficient of a Laurent series $A(s)$ around $s_0$. For each fixed automorphic copy of an inducing representation, the space of sections with level $K_\fin$ and the scalar $K_\infty$-type $-k$ is one-dimensional, so only scalar polynomial multiples of its standard section remain. In particular, we must have
	\[ Q_{ij}(s) = \psi^{s+(n-1)/2}q_{ij}(s)\varphi_{f_i\otimes v_j} \]
	for some polynomial
	\[ q_{ij}(s) = \sum_{l=0}^{\deg(q_{ij})}a_{ij;l}\cdot(s-s_0)^l. \]
	Recall that $M(w,s)$ acts on $\psi^{s+(n-1)/2}\varphi_{f_i\otimes v_j}$ by a meromorphic scalar, say
	\[ M(w,s)\psi^{s+(n-1)/2}\varphi_{f_i\otimes v_j} = m_{ij}(s)\psi^{-s+(n-1)/2}\varphi_{f_i\otimes v_j}. \]
	Write
	\[ m_{ij}(s)q_{ij}(s) = \sum_{l\gg-\infty}^\infty b_{ij;l}\cdot(s-s_0)^l. \]
	According to the constant term formula Theorem \ref{thm:LanglandsConstantTerm} for $m\in M(\A)$ and $a\in M(\A)\cap T(\A)$ we have
	\begin{multline*}
		[(s-s_0)^{N_{ij}}]_{E(\cdot,Q_{ij}(s))_P}(am) = \\
		\psi(a)^{k/2}\sum_{l=0}^{N_{ij}}\frac{\ln(\psi(a))^l}{l!}a_{ij;N_{ij}-l}\cdot\varphi_{f_i\otimes v_j}(m) + \\
		+ \psi(a)^{n-1-k/2}
		\sum_{l=0}^{\infty}\frac{(-\ln(\psi(a)))^l}{l!}b_{ij;N_{ij}-l}\cdot\varphi_{f_i\otimes v_j}(m).
	\end{multline*}
	Note that $(\phi_F)_P(am) = \psi(a)^{k/2}(\phi_F)_P(m)$. So Independence of the inducing vectors and the absence of logarithmic terms force
	\[ a_{ij;l}=0\ (l<N_{ij}),\qquad a_{ij;N_{ij}}=c_{ij},\qquad b_{ij;l}=0\ (l\leq N_{ij}). \]
	If $c_{ij}\ne0$, then $q_{ij}$ has order $N_{ij}$ at $s_0$ whereas $m_{ij}q_{ij}$ has order at least $N_{ij}+1$. Thus $m_{ij}(s_0)=0$. Theorem \ref{thm:EisensteinseriesClassicalIsAdelic} now constructs every required $E_{f_i\otimes v_j}$. Their linear combination has the same boundary as $F$ and the difference is both cuspidal and orthogonal to cusp forms, hence zero.
\end{proof}

For a non-zero holomorphic value in Theorem \ref{thm:EisensteinseriesClassicalIsAdelic} the only normalized cuspidal exponent is $s_0=-(n-k-1)/2$. Langlands' square-integrability criterion therefore gives finite Petersson norm precisely when $s_0<0$, equivalently $k\leq n-2$. These are cases (iii) and (iv). See \cite[Section I.4.11]{MoeglinWaldspurger} and \cite{Langlands}.

\begin{prp}\label{prp:ParameterOfEisenstein1}
Let $f\in\mathrm S_k(\SL_2(\Z))$ be a normalized eigenform with $k<n-2$ and let $v$ be a simultaneous eigenvector for $\mathcal{H}(\Sort_W)$ with standard parameter
\[ \psi=\psi'\oplus\pi_f[d],\qquad d=n-k-2>0. \]
Assume $L(\frac{1}{2},\pi_f\otimes\pi_i)\ne0$ for every other summand $\pi_i[d_i]$ with $d_i=d+1$. Then $E_{f\otimes v}\in\mathrm M_k(\SGamma_n,1)$ is a non-zero Hecke eigenform, its adelization lies in the discrete spectrum and its standard parameter is
\[ \psi'\oplus\pi_f[d+2]=\psi'\oplus\pi_f[n-k]. \]
\end{prp}
\begin{proof}
The hypothesis is precisely case (iv) of Theorem \ref{thm:EisensteinseriesClassicalIsAdelic}. It gives holomorphicity and the non-zero boundary value $f\otimes v$. Since $s_0=-(d+1)/2<0$, the preceding square-integrability criterion applies. A square-integrable automorphic form which is $K$-finite and $Z(\mathfrak g)$-finite belongs to the discrete spectrum. Equivalently, it is a value of the negative-exponent Eisenstein family, or a residual contribution via the functional equation (see \cite[Chapter 10]{GetzHahn} and \cite{Langlands}).

The spherical line of each $I(\pi_{f,p}\otimes\pi_{v,p},s)$ is one-dimensional. Since Eisenstein summation commutes with Hecke operators, specializing this line at $s_0$ shows that $E_{f\otimes v}$ is an eigenform. Write $a_p=s_p(\pi_f)$ and $b_p=s_p(\pi_v)$. Restriction of the standard representation of $\Sort_{n+2}(\C)$ to $\widehat M$ gives the standard Satake parameter
\[ p^{-s_0}a_p\ \oplus\ \Std_{\Sort_{n-2}}(b_p)\ \oplus\ p^{s_0}a_p^{\vee}. \]
The representation $\pi_f$ is self-dual. Thus the two new blocks are $a_pp^{(d+1)/2}$ and $a_pp^{-(d+1)/2}$. By assumption, $b_p$ contains the block
\[ a_p\otimes[d]_p=\bigoplus_{j=0}^{d-1}a_pp^{j-(d-1)/2}. \]
Together they form $a_p\otimes[d+2]_p$. This proves the asserted parameter at every finite prime. The corresponding addition of infinitesimal characters gives the same parameter at infinity. Hence, the standard parameter of $E_{f\otimes v}$ is $\psi'\oplus\pi_f[n-k]$.
\end{proof}

Similarly, we have
\begin{prp}\label{prp:ParameterOfEisenstein2}
	Let $f\in\mathrm S_{n-2}(\SL_2(\Z))$ be a normalized eigenform and let $v$ be a simultaneous eigenvector for $\mathcal{H}(\Sort_W)$ with standard parameter $\psi$. Assume $L(\frac{1}{2},\pi_f\otimes\pi_i)\not=0$ for every summand $\pi_i[1]$ of $\psi$. Then $E_{f\otimes v}\in\mathrm M_k(\SGamma_n,1)$ is a non-zero Hecke eigenform, its adelization lies in the discrete spectrum and its standard parameter is
	\[ \psi\oplus\pi_f[2]. \]
\end{prp}

In the following we list some interesting examples that satisfy the conditions of Theorem \ref{thm:EisensteinseriesClassicalIsAdelic}.

\begin{thm}\label{thm:SingularWeightEisensteinSeries}
	Suppose that (H1) is true. Let $f\in\mathrm{S}_{n/2-1}(\SL_2(\Z))$ be a normalized eigenform. There exists a $v_f^1\in\C[I\!I_{n-2,0}^1]$ and $v_f^{\det}\in\C[I\!I_{n-2,0}^{\det}]$ with Arthur parameter $\psi = \pi_f[n/2-1]$. The Eisenstein series $E_{f\otimes v_f^\chi}(Z)\in\mathrm{M}_{n/2-1}(\Gamma_n,\chi)$ is of singular weight. We have
	\begin{multline*}
		\mathrm{M}_{n/2-1}(\SGamma_n,1) = \underbrace{\C\Psi_{I\!I_{n,2}}^{\mathrm{add}}(1)}_{\subset\mathcal{E}_{n/2-1}^0(\Gamma_{n},1)} \oplus \\
		\oplus\bigoplus_{\substack{f\in\mathrm{S}_{n/2-1}(\SL_2(\Z)) \\ \text{normalized eigenform}}}\left(\underbrace{\C E_{f\otimes v_f^1}}_{\subset\mathcal{E}_{n/2-1}^1(\Gamma_{n},1)} \oplus \underbrace{\C E_{f\otimes v_f^{\det}}}_{\subset\mathcal{E}_{n/2-1}^1(\Gamma_{n},\det)}\right).
	\end{multline*}
	If $L(\frac{1}{2},\pi_f,\Std_{\GL_2}) \not=0$, then $v_f^\chi\in\C[I\!I_{n-2,0}^\chi]$ exists without the assumption of (H1) and $\vartheta_{I\!I_{n-2,0}}^{(n/2-1)}(v_f^1)$ is the Ikeda-lift of $f$.
\end{thm}
\begin{proof}
	The existence of the vectors $v_f^\chi\in\C[I\!I_{n-2,0}^\chi]$ with Arthur parameter $\psi = \pi_f[n/2-1]$ follows from Arthur's multiplicity formula \cite[Theorem 8.5.8]{ChenevierLannes}. The computation for these parameters in particular was done in \cite[Proposition 5]{Muller}. At $k=n/2-1$, case (iv) of Theorem \ref{thm:EisensteinseriesClassicalIsAdelic} requires the block $\pi_f[n/2-1]$. Its dimension is already $n-2$, so there can be no other summands and no central-value cancellation. Proposition \ref{prp:EisensteinBasis} therefore gives $\mathcal{E}_{n/2-1}^1(\SGamma_n,1)$ as claimed. Singular-weight cusp forms vanish, while $\mathcal{E}_{n/2-1}^0(\SGamma_n,1)$ is spanned by $\Psi_{I\!I_{n,2}}^{\mathrm{add}}(1)$ as explained in the beginning of Section \ref{sec:HolomorphicModularForms}. This proves the decomposition. The assertion concerning the Ikeda lift, with a suitable normalization of $v_f^1$, is \cite[Theorem 7.3.2]{ChenevierLannes}. The existence of $v_f^{\det}$ then follows by the same argument as Proposition \ref{prp:Psi12}.
\end{proof}
We remark that it is conjectured (see e.g.\ \cite{ConreyFarmer}) that $L(\frac{1}{2},\pi_f,\Std_{\GL_2}) \not= 0$ whenever $f\in\mathrm{S}_k(\SL_2(\Z))$ with $k\equiv0\bmod4$. This has been verified up to $k\leq 2000$ \cite{ConreyFarmer, FarmerJames}.

In particular, we obtain a second proof of Corollary \ref{cor:SingularWeightFormsn=26} and 
\begin{thm}\label{thm:Phi12IsEisensteinSeries}
	We have
	\[ \Psi_{12} = E_{\Delta\otimes v_\Delta^{1}}\qquad\text{and}\qquad\Phi_{12} = E_{\Delta\otimes v_\Delta^{\det}}. \]
\end{thm}

Similar to the previous proposition we have

\begin{prp}
	Suppose that (H1) is true. Let $f\in\mathrm{S}_k(\SL_2(\Z))$ be a normalized eigenform with $n/2-1 < k < n-2$ and $k> 2/3(n-1)$ if $k\equiv 2\bmod4$. There exists a $v_f\in\C[\widetilde{I\!I}_{n-2,0}]$ with Arthur parameter $\psi = [1] \oplus [2k-n+1] \oplus \pi_f[n-k-2]$. The Eisenstein series $E_{f\otimes v_f}(Z)\in\mathrm{M}_{k}(\SGamma_n,1)$. If $n-2>3(n-k-2)$ or $L(\frac{1}{2},\pi_f,\Std_{\GL_2})\not=0$, then $v_f$ lies in $\C[I\!I_{n-2,0}^{1}]$ and exists without the assumption of (H1). Its degree-$(n-k-2)$ theta lift is the Ikeda lift of $f$, after normalization (see \cite[Theorem 7.3.2]{ChenevierLannes}).
\end{prp}

We list some further examples of eigenforms $f\in\mathrm{S}_k(\SL_2(\Z))$ and parameters $\psi$ for which there exists an eigenvector $v\in\C[\widetilde{I\!I}_{n-2,0}]$ with Arthur parameter $\psi$ such that $E_{f\otimes v}\in\mathrm{M}_k(\SGamma_n,1)$ according to Theorem \ref{thm:EisensteinseriesClassicalIsAdelic}. All elliptic cusp forms are normalized eigenforms. We assume (H1) when using the multiplicity formula to assert existence of an inducing parameter.
\begin{enumerate}[1)]
	\item Let $f\in\mathrm{S}_k(\SL_2(\Z))$ with $n/2-1 \leq k < n-2$ and choose integers $k_2<\hdots,k_m = k$ and set $d_m = n-k-2$, $d_i = k_{i+1} - k_i - d_{i+1}$ for $i=2,\hdots,m-1$ and $k_1 = d_1 = \frac{k_2-d_2}{2}$. Choose eigenforms $f_i\in\mathrm{S}_{k_i}(\SL_2(\Z))$ for $i=1,\hdots,m-1$ and set $f_m=f$ and $\psi = \bigoplus_{i=1}^m\pi_{f_i}[d_i]$. Here we require $d_i>0$ and $d_i\equiv0\bmod4$ for all $i$. These congruences make the multiplicity character trivial. The interval relations above give the required infinitesimal character, so existence follows from \cite[Theorem 8.5.8]{ChenevierLannes} under (H1).
	\item Let $f\in\mathrm{S}_{n-2}(\SL_2(\Z))$ with $L(\frac{1}{2},\pi_f,\Std_{\GL_2}) \not=0$ and $\psi = [1]\oplus[n-3]$.
	\item Let $f\in\mathrm{S}_{n}(\SL_2(\Z))$ be a normalized eigenform and $\psi = [1]\oplus[n-3]$.\label{exp:weightnExample1}
	\item Let $f\in\mathrm{S}_{n-2}(\SL_2(\Z))$ with $L(\frac{1}{2},\pi_f,\Std_{\GL_2}) \not=0$, $n/2-1 < k < n-2$ and $k\geq 2/3(n-1)$ if $k\equiv 2\bmod4$. Choose $g\in\mathrm{S}_k(\SL_2(\Z))$ and set $\psi = [1]\oplus[2k-n+1]\oplus\pi_g[n-k-2]$.
	\item Let $f\in\mathrm{S}_{n}(\SL_2(\Z))$ be a normalized eigenform, $n/2-1 < k < n-2$ and $k\geq 2/3(n-1)$ if $k\equiv 2\bmod4$. Choose $g\in\mathrm{S}_k(\SL_2(\Z))$ and set $\psi = [1]\oplus[2k-n+1]\oplus\pi_g[n-k-2]$.\label{exp:weightnExample2}
\end{enumerate}

For the weight-$n$ examples, $n\equiv2\bmod8$ implies
\[ \varepsilon(\pi_f)=(-1)^{n/2}=-1,\qquad L(\tfrac{1}{2},\pi_f,\Std_{\GL_2})=0. \]
Thus their $[1]$-summand gives precisely the vanishing required in case (ii).

\section{Geometric applications}\label{sec:GeneralType}

Let $X$ be a smooth projective variety with canonical bundle $K_X$. The \emph{Kodaira dimension} of $X$ is the integer $\kappa(X)$ such that $\dim (H^0(X,K_X^{\otimes d})) \sim d^{\kappa(X)}$. If $\dim (H^0(X,K_X^{\otimes d})) = 0$ for all $d\in\Z_{>0}$, we set $\kappa(X)=-\infty$. Otherwise, one always has $\kappa(X)\in\{0,\hdots,\dim(X)\}$. When $\kappa(X) = \dim(X)$, we say that $X$ is of \emph{general type}. Suppose that $X$ is not smooth nor projective. Let $X^\circ$ be the smooth locus of $X$ and $\overline{X^\circ}$ a smooth compactification. Then $\kappa(\overline{X^\circ})$ does not depend on the chosen compactification and we may define $\kappa(X) \define \kappa(\overline{X^\circ})$.

We use the following form of the low-weight cusp-form argument of Gritsenko, Hulek and Sankaran. The version with $d>1$ follows by taking powers in the same argument.
\begin{thm}[{cf.\ \cite[Theorem 1.1 and Section 2]{GritsenkoHulekSankaran2}}]\label{thm:GeneralTypeCriterion}
Let $L$ be an integral lattice of signature $(n,2)$, $n\geq9$ and let $\Gamma$ be a subgroup of finite index of $\Ort^+(L)$. The modular variety $\Gamma\backslash\mathcal{H}$ is of general type if there exists a character $\chi$ and a non-zero cusp form $F_k\in\mathrm{S}_k(\Gamma,\chi)$ of weight $k < dn$ that vanishes along the cusps and along the ramification divisor $R$ of the projection $\pi:\mathcal{H} \to \Gamma\backslash\mathcal{H}_n$ with order at least $d$.
\end{thm}

In particular, they showed
\begin{thm}[cf.\ Theorem 1.1 and 4.4 in \cite{GritsenkoHulekSankaran3}]
	Let $n = 8\ell+2$. Then
	\[ \kappa(\Gamma_n\backslash\mathcal{H}) = \begin{cases}
		-\infty & \text{if } \ell=0,1,2 \\
		\dim(\Gamma_n\backslash\mathcal{H}) = n & \text{if } \ell \geq5.
	\end{cases} \]
\end{thm}
In this case of $\Gamma_n$, the ramification divisor $R$ is given by
\[ \Gamma_n\backslash\sumstack{\lambda\in I\!I_{n,2} \\ (\lambda,\lambda) = 2}\lambda^\bot \subset\Gamma_{n}\backslash\mathbf{D}, \]
where $\lambda^\bot$ is the fixed point divisor of the reflection $\sigma_\lambda$ along $\lambda$ (cf.\ \cite[Section 2]{GritsenkoHulekSankaran2}). On the other hand, if $F\in\mathrm{M}_k(\Gamma_n,\det)$, then
\[ F|_k[\sigma_\lambda]=-F \]
so that $F$ vanishes on the fixed divisor of $\sigma_\lambda$ since $j(\sigma_\lambda,Z) = 1$ if $[Z+z'+\q(Z)z]\in\lambda^\bot$. Hence, $F$ vanishes along $R$. We can now show
\begin{thm}\label{thm:GeneralType26}
	Let $n = 26 = 8\cdot3+2$. Then
	\[ \kappa(\Gamma_{26}\backslash\mathcal{H}_{26}) = \dim(\Gamma_{26}\backslash\mathcal{H}_{26}) = 26. \]
\end{thm}
\begin{proof}
	Consider $\Phi_{12}\in\mathrm{M}_{12}(\Gamma_{26},\det)$. It vanishes along $R$ and $\Phi(\Phi_{12}) =\Delta\otimes([\Lambda^+] - [\Lambda^-])$ so that it vanishes along all cusps except for the cusp corresponding to $[\Lambda]$. It follows that $\Phi_{12}^2\in\mathrm{M}_{24}(\Gamma_{26},1)$ vanishes along $R$ and all cusps except $[\Lambda]$ with order at least $2$. We construct a modular form $F\in\mathrm{M}_{26}(\Gamma_{26},1)$ of weight $26$ and order of vanishing in $[\Lambda]$ at least $2$, so that $F\cdot \Phi_{12}^2$ satisfies the conditions of Theorem \ref{thm:GeneralTypeCriterion} for $k=50<2\cdot26$ and $d=2$: \\
	We write $Z = (\tau',\tau,\mathfrak{z})$ and
	\[ F(Z) = f_0(\tau) + \sum_{m=1}^\infty\langle f_m(\tau),\overline{\vartheta_{\Lambda(m)}(\tau,\mathfrak{z})}\rangle e(m\tau') \]
	where $f_0 = \Phi_{\Lambda}(F)\in\mathrm{M}_{26}(\SL_2(\Z))$ and $f_m\in\mathrm{M}_{14}(\omega_{\Lambda(m)})$ is the modular form in the theta decomposition of the $m$-th Fourier--Jacobi coefficient. The vanishing order of $F$ along $[\Lambda]$ is the smallest integer $m$, such that $f_{m}$ does not vanish. Since $\Lambda$ is unimodular, $\mathrm{M}_{14}(\omega_{\Lambda(1)}) = \mathrm{M}_{14}(\SL_2(\Z))$ and it follows from the invariance of $F$ under $\sigma_{\rho_{\Lambda} - \rho_{\Lambda}'}$ that $a_{f_1}(0) = a_{f_0}(1)$. Define
	\begin{multline*}
		\Obs_2(\mathrm{M}_k(\Gamma,1),[D]) \\
		\define \{(g_0,g_1)\in\mathrm{M}_k(\SL_2(\Z))\oplus\mathrm{M}_{k-n/2+1}(\SL_2(\Z))\mid a_{g_1}(0) = a_{g_0}(1)\}
	\end{multline*}
	so that $F$ vanishes with order at least $2$ if it lies in the kernel of the natural projection $\mathrm{M}_{26}(\Gamma_{26},1)\to\Obs_2(\mathrm{M}_{26}(\Gamma_{26},1),[\Lambda])$. In our case 
	\[ \dim(\Obs_2(\mathrm{M}_{26}(\Gamma_{26},1),[\Lambda])) = 2 + 1 - 1 = 2. \]
	On the other hand, let $f_{26}$ and $g_{18}$ be the normalized generators of $\mathrm S_{26}(\SL_2(\Z))$ and $\mathrm S_{18}(\SL_2(\Z))$. There are two independent vectors $v_0,v_{18}\in\C[I\!I_{24,0}]$ with parameters
	\[ [23]\oplus[1],\qquad \pi_{g_{18}}[6]\oplus[11]\oplus[1]. \]
	The first is the constant function on lattice classes. The second exists unconditionally by \cite[Theorem 7.3.2]{ChenevierLannes} since $24>3\cdot6$. As $L(\frac{1}{2},\pi_{f_{26}},\Std_{\GL_2})=0$, case (ii) gives two holomorphic Eisenstein series with boundary values $f_{26}\otimes v_0$ and $f_{26}\otimes v_{18}$. Together with the additive lift $\Psi_{I\!I_{26,2}}^{\mathrm{add}}(E_{14})$, whose zero-dimensional constant term is non-zero, these are linearly independent (see \cite[Theorem 14.3]{Borcherds98} for the additive lift). Thus $\dim\mathrm M_{26}(\Gamma_{26},1)\geq3$ and the obstruction map has a non-zero kernel.
\end{proof}
Replacing $\Phi_{12}$ by the Eisenstein series of singular weight constructed in Theorem \ref{thm:SingularWeightEisensteinSeries}, we can show
\begin{thm}\label{thm:GeneralType34}
	Let $n=8\ell+2$ with $\ell\geq4$ and if $\ell>500$ assume (H1). Suppose that the number of positive-definite even unimodular lattices of rank $8\ell$ which are chiral is smaller than
	\[ 24(n-1)(n/2+1)^{n-2}\cdot\frac{2^{-n/2+1}}{n!}\cdot\frac{B_2\cdot B_4\cdot\hdots\cdot B_n}{n!!}\cdot B_{n/2+1}. \]
	Then
	\[ \kappa(\Gamma_{n}\backslash\mathcal{H}_{n}) = \dim(\Gamma_{n}\backslash\mathcal{H}_{n}) = n. \]
\end{thm}
\begin{proof}
	We proceed as in the previous theorem. Let $f\in\mathrm{S}_{4\ell}(\SL_2(\Z))$ be a simultaneous eigenform. By Theorem \ref{thm:SingularWeightEisensteinSeries}, there exists a holomorphic Eisenstein series $E_{f\otimes v_f^{\det}}\in\mathrm{M}_{4\ell}(\Gamma_{n},\det)$. Let $d$ tend to infinity through even integers. Then $E_{f\otimes v_f^{\det}}^d\in\mathrm{M}_{d(n/2-1)}(\Gamma_{n},1)$ has order of vanishing along $R$ and along all achiral cusps at least $d$. We show the existence of a form $F\in\mathrm{M}_{k}(\Gamma_n,1)$ with order of vanishing along the chiral cusps at least $d$ for $k\define d(n/2+1)-2$, so that $F\cdot E_{f\otimes v_f^{\det}}^d$ satisfies the condition of Theorem \ref{thm:GeneralTypeCriterion}: 

	Put $r=n-2=8\ell$ and $w=k-r/2$. Let $\Obs_d(\mathrm M_k(\Gamma_n,1),[D])$ be the image of the map taking the Fourier--Jacobi coefficients of indices $0,\ldots,d-1$. Theta decomposition embeds it in
	\[ \mathrm M_k(\SL_2(\Z))\oplus\bigoplus_{m=1}^{d-1}\mathrm M_w(\omega_{D(m)}). \]
	For the even weights under consideration, the subspace of the Weil representation on which modular forms of weight $w$ are supported has dimension
	\[ r_m=\frac{|D/mD|+|(D/mD)[2]|}{2}=\frac{m^r+\gcd(m,2)^r}{2}. \]
	The finite-image dimension formula of \cite[Theorem 6.1]{FreitagDimensionFormulas} gives uniformly in $m$
	\[ \dim(\mathrm M_w(\omega_{D(m)}))=\frac{w}{12}r_m+O(r_m). \]
	Indeed, each correction term is bounded by a constant times the dimension of the representation. Consequently,
	\begin{align*}
	\dim(\Obs_d(\mathrm M_k(\Gamma_n,1),[D]))
	&\leq\frac{k}{12}+\frac{w}{12}\sum_{m=1}^{d-1}r_m+O\left(1+\sum_{m=1}^{d-1}r_m\right)\\
	&\leq\frac{n/2+1}{24(n-1)}d^n+O(d^{n-1}).
	\end{align*}
	This uniform estimate is needed because the index $m$ varies with $d$. By \cite[Section 3.3]{GritsenkoHulekSankaran1},
	\[ \dim(\mathrm S_k(\Gamma_n,1)=C_nk^n+O(k^{n-1})),\quad
	C_n=\frac{2^{-n/2+1}}{n!}\cdot\frac{B_2\cdot B_4\cdot\hdots\cdot B_n}{n!!}\cdot\frac{B_{n/2+1}}{n/2+1}. \]
	If $h_{\mathrm{ch}}$ is the number of chiral classes, the numerical hypothesis is exactly
	\[ C_n(n/2+1)^n>h_{\mathrm{ch}}\frac{n/2+1}{24(n-1)}. \]
	For all sufficiently large even $d$, the map from $\mathrm M_k(\Gamma_n,1)$ to the direct sum of the $h_{\mathrm{ch}}$ obstruction spaces therefore has a non-zero kernel. Choose $F$ in this kernel. Its product with $E_{f\otimes v_f^{\det}}^d$ has weight $dn-2<dn$ and the required vanishing along the reflective and isotropic boundary divisors. Theorem \ref{thm:GeneralTypeCriterion} completes the proof.
\end{proof}

We remark that for $n=34$, i.e.\ $\ell=4$ we have
\[ 24(n-1)(n/2+1)^{n-2}\cdot\frac{2^{-n/2+1}}{n!}\cdot\frac{B_2\cdot B_4\cdot\hdots\cdot B_n}{n!!}\cdot B_{n/2+1} \approx 1.18\cdot10^{33}, \]
which far exceeds the lower bound $10^7$ of the number of positive-definite even unimodular lattices of rank $32$ without roots obtained by King (see \cite[Corollary 5]{King}).

\bibliographystyle{amsplain}
\bibliography{references}

\providecommand{\bysame}{\leavevmode\hbox to3em{\hrulefill}\thinspace}
\providecommand{\MR}{\relax\ifhmode\unskip\space\fi MR }
\providecommand{\MRhref}[2]{%
  \href{http://www.ams.org/mathscinet-getitem?mr=#1}{#2}
}
\providecommand{\href}[2]{#2}
\begin{thebibliography}{10}

\bibitem{Arthur}
James Arthur, \emph{The endoscopic classification of representations:
  Orthogonal and symplectic groups}, American Mathematical Society Colloquium
  Publications, vol.~61, American Mathematical Society, Providence, RI, 2013.
  \MR{3135650}

\bibitem{AGIKMS}
Hiraku Atobe, Wee~Teck Gan, Atsushi Ichino, Tasho Kaletha, Alberto M\'inguez,
  and Sug~Woo Shin, \emph{Local {I}ntertwining {R}elations and {C}o-tempered
  {$A$}-packets of {C}lassical {G}roups}, preprint, arXiv: 2410.13504, 2025.

\bibitem{BorcherdsMM}
Richard~E. Borcherds, \emph{Monstrous moonshine and monstrous {L}ie
  superalgebras}, Invent. Math. \textbf{109} (1992), no.~2, 405--444.
  \MR{1172696}

\bibitem{Borcherds98}
\bysame, \emph{Automorphic forms with singularities on {G}rassmannians},
  Invent. Math. \textbf{132} (1998), no.~3, 491--562. \MR{1625724}

\bibitem{BorcherdsFreitagWeissauer}
Richard~E. Borcherds, E.~Freitag, and R.~Weissauer, \emph{A {S}iegel cusp form
  of degree 12 and weight 12}, vol. 494, 1998, Dedicated to Martin Kneser on
  the occasion of his 70th birthday, pp.~141--153. \MR{1604476}

\bibitem{ChaudouardLaumon1}
Pierre-Henri Chaudouard and G\`erard Laumon, \emph{Le lemme fondamental
  pond\'er\'e. {I}. {C}onstructions g\'eom\'etriques}, Compos. Math.
  \textbf{146} (2010), no.~6, 1416--1506. \MR{2735371}

\bibitem{ChaudouardLaumon2}
\bysame, \emph{Le lemme fondamental pond\'er\'e. {II}. \'enonc\'es
  cohomologiques}, Ann. of Math. (2) \textbf{176} (2012), no.~3, 1647--1781.
  \MR{2979859}

\bibitem{ChenevierLannes}
Ga\"etan Chenevier and Jean Lannes, \emph{Automorphic forms and even unimodular
  lattices}, Ergebnisse der Mathematik und ihrer Grenzgebiete. 3. Folge. A
  Series of Modern Surveys in Mathematics [Results in Mathematics and Related
  Areas. 3rd Series. A Series of Modern Surveys in Mathematics], vol.~69,
  Springer, Cham, 2019, Kneser neighbors of Niemeier lattices, Translated from
  the French by Reinie Ern\'e. \MR{3929692}

\bibitem{Cogdell1}
James~W. Cogdell, \emph{Lectures on {$L$}-functions, converse theorems, and
  functoriality for {${\rm GL}_n$}}, Lectures on automorphic {$L$}-functions,
  Fields Inst. Monogr., vol.~20, Amer. Math. Soc., Providence, RI, 2004,
  pp.~1--96. \MR{2071506}

\bibitem{Cogdell2}
\bysame, \emph{{$L$}-functions and functoriality}, Automorphic forms and
  {$L$}-functions, Adv. Lect. Math. (ALM), vol.~30, Int. Press, Somerville, MA,
  2014, \url{https://people.math.osu.edu/cogdell.1/cimpa-www.pdf}, pp.~1--47.
  \MR{3307911}

\bibitem{ConreyFarmer}
John~B. Conrey and David~W. Farmer, \emph{Hecke operators and the nonvanishing
  of {$L$}-functions}, Topics in number theory ({U}niversity {P}ark, {PA},
  1997), Math. Appl., vol. 467, Kluwer Acad. Publ., Dordrecht, 1999,
  pp.~143--150. \MR{1691315}

\bibitem{Deligne}
Pierre Deligne, \emph{La conjecture de {W}eil. {I}}, Publ. Math. Inst. Hautes
  {\'E}tudes Sci. \textbf{43} (1974), 273--307.

\bibitem{EichlerZagier}
Martin Eichler and Don Zagier, \emph{The theory of {J}acobi forms}, Progress in
  Mathematics, vol.~55, Birkh\"auser Boston, Inc., Boston, MA, 1985.
  \MR{781735}

\bibitem{FarmerJames}
David~W. Farmer and Kevin James, \emph{The irreducibility of some level 1
  {H}ecke polynomials}, Math. Comp. \textbf{71} (2002), no.~239, 1263--1270.
  \MR{1898755}

\bibitem{FrankeSchwermer}
Jens Franke and Joachim Schwermer, \emph{A decomposition of spaces of
  automorphic forms, and the {E}isenstein cohomology of arithmetic groups},
  Math. Ann. \textbf{311} (1998), no.~4, 765--790. \MR{1637980}

\bibitem{FreitagDimensionFormulas}
Eberhard Freitag, \emph{Dimension formulae for vector valued automorphic
  forms}, 2012, Available at
  \url{https://www.mathi.uni-heidelberg.de/~freitag/preprints/riemann-roch.pdf}.

\bibitem{GetzHahn}
Jayce~R. Getz and Heekyoung Hahn, \emph{An introduction to automorphic
  representations---with a view toward trace formulae}, Graduate Texts in
  Mathematics, vol. 300, Springer, Cham, 2024. \MR{4738301}

\bibitem{GritsenkoNikulin1}
Valeri~A. Gritsenko and Viacheslav~V. Nikulin, \emph{Automorphic forms and
  {L}orentzian {K}ac-{M}oody algebras. {I}}, Internat. J. Math. \textbf{9}
  (1998), no.~2, 153--199. \MR{1616925}

\bibitem{GritsenkoNikulin2}
\bysame, \emph{Automorphic forms and {L}orentzian {K}ac-{M}oody algebras.
  {II}}, Internat. J. Math. \textbf{9} (1998), no.~2, 201--275. \MR{1616929}

\bibitem{GritsenkoHulekSankaran1}
Valery Gritsenko, Klaus Hulek, and Gregory~K. Sankaran, \emph{The
  {H}irzebruch-{M}umford volume for the orthogonal group and applications},
  Doc. Math. \textbf{12} (2007), 215--241. \MR{2350289}

\bibitem{GritsenkoHulekSankaran2}
\bysame, \emph{The {K}odaira dimension of the moduli of {K3} surfaces}, Invent.
  Math. \textbf{169} (2007), 519--567.

\bibitem{GritsenkoHulekSankaran3}
\bysame, \emph{Hirzebruch--mumford proportionality and locally symmetric
  varieties of orthogonal type}, Doc. Math. \textbf{13} (2008), 1--19.

\bibitem{GritsenkoFourierJacobi}
Valery~A. Gritsenko, \emph{Fourier-{J}acobi functions in {$n$} variables}, Zap.
  Nauchn. Sem. Leningrad. Otdel. Mat. Inst. Steklov. (LOMI) \textbf{168}
  (1988), 32--44, 187--188. \MR{982481}

\bibitem{Huybrechts}
Daniel Huybrechts, \emph{Lectures on {K}3 surfaces}, Cambridge Studies in
  Advanced Mathematics, vol. 158, Cambridge University Press, Cambridge, 2016.
  \MR{3586372}

\bibitem{King}
Oliver~D. King, \emph{A mass formula for unimodular lattices with no roots},
  Math. Comp. \textbf{72} (2003), no.~242, 839--863. \MR{1954971}

\bibitem{KriegRoemerSchaps}
Aloys Krieg, Hannah R{\"o}mer, and Felix Schaps, \emph{Hecke theory for
  {$\mathrm{SO}^{+}(2,n+2)$}}, J. Number Theory \textbf{262} (2024), 454--470.

\bibitem{Langlands}
Robert~P. Langlands, \emph{On the functional equations satisfied by
  {E}isenstein series}, Lecture Notes in Mathematics, vol. 544,
  Springer-Verlag, Berlin-New York, 1976. \MR{579181}

\bibitem{MoeglinWaldspurger}
Colette M{\oe}glin and Jean-Loup Waldspurger, \emph{Spectral decomposition and
  {E}isenstein series}, Cambridge Tracts in Mathematics, vol. 113, Cambridge
  University Press, Cambridge, 1995, Une paraphrase de l'\'Ecriture [A
  paraphrase of Scripture]. \MR{1361168}

\bibitem{MoeglinWaldspurger1}
\bysame, \emph{Stabilisation de la formule des traces tordue. {V}ol. 1},
  Progress in Mathematics, vol. 316, Birkh\"auser/Springer, Cham, 2016.
  \MR{3823813}

\bibitem{MoeglinWaldspurger2}
\bysame, \emph{Stabilisation de la formule des traces tordue. {V}ol. 2},
  Progress in Mathematics, vol. 317, Birkh\"auser/Springer, Cham, 2016.
  \MR{3823814}

\bibitem{Muller}
Manuel K.-H. M\"uller, \emph{{L}inear independence of theta series of
  positive-definite even unimodular lattices}, Manuscripta Math. \textbf{177}
  (2026), no.~5, 69.

\bibitem{thesis_mmueller}
Manuel K.-H. Müller, \emph{Generating systems for modular forms for the {W}eil
  representation and {H}ecke operators for orthogonal modular forms}, Ph.D.
  thesis, Technische {U}niversit\"at {D}armstadt, 2024, Available at
  \url{https://tuprints.ulb.tu-darmstadt.de/29551/20/thesis_mmueller.pdf}.

\bibitem{NebeVenkov}
Gabriele Nebe and Boris Venkov, \emph{On {S}iegel modular forms of weight 12},
  J. Reine Angew. Math. \textbf{531} (2001), 49--60. \MR{1810116}

\bibitem{Oda}
Takayuki Oda, \emph{On modular forms associated with indefinite quadratic forms
  of signature {$(2, n-2)$}}, Math. Ann. \textbf{231} (1977/78), no.~2,
  97--144. \MR{466026}

\bibitem{RallisSchiffmann}
Stephen Rallis and G\`erard Schiffmann, \emph{On a relation between
  {$\widetilde {\rm SL}\sb{2}$}\ cusp forms and cusp forms on tube domains
  associated to orthogonal groups}, Trans. Amer. Math. Soc. \textbf{263}
  (1981), no.~1, 1--58. \MR{590410}

\bibitem{DLMF}
\emph{\it {NIST} {D}igital {L}ibrary of {M}athematical {F}unctions},
  \url{https://dlmf.nist.gov/}, Release 1.2.8 of 2026-09-15, F.~W.~J. Olver,
  A.~B. {Olde Daalhuis}, D.~W. Lozier, B.~I. Schneider, R.~F. Boisvert, C.~W.
  Clark, B.~R. Miller, B.~V. Saunders, H.~S. Cohl, and M.~A. McClain, eds.

\bibitem{ScheithauerInventiones}
Nils~R. Scheithauer, \emph{On the classification of automorphic products and
  generalized {K}ac-{M}oody algebras}, Invent. Math. \textbf{164} (2006),
  no.~3, 641--678. \MR{2221135}

\bibitem{ScheithauerIMRN}
\bysame, \emph{The {W}eil representation of {${\rm SL}_2(\Bbb Z)$} and some
  applications}, Int. Math. Res. Not. IMRN (2009), no.~8, 1488--1545.
  \MR{2496771}

\bibitem{ScheithauerCompositio}
\bysame, \emph{Automorphic products of singular weight}, Compos. Math.
  \textbf{153} (2017), no.~9, 1855--1892. \MR{3705279}

\bibitem{Shahidi}
Freydoon Shahidi, \emph{Eisenstein series and automorphic {$L$}-functions},
  American Mathematical Society Colloquium Publications, vol.~58, American
  Mathematical Society, Providence, RI, 2010. \MR{2683009}

\bibitem{Shin}
Sug~Woo Shin, \emph{Weak transfer from classical groups to general linear
  groups}, Essent. Number Theory \textbf{3} (2024), no.~1, 19--62. \MR{4752712}

\bibitem{Waldspurger}
Jean-Loup Waldspurger, \emph{{\`A} propos du lemme fondamental pond\'er\'e{}
  tordu}, Math. Ann. \textbf{343} (2009), no.~1, 103--174. \MR{2448443}

\bibitem{Weil}
Andr\'e Weil, \emph{Sur certains groupes d'op\'erateurs unitaires}, Acta Math.
  \textbf{111} (1964), 143--211. \MR{165033}

\bibitem{Weissauer}
Rainer Weissauer, \emph{Stabile {M}odulformen und {E}isensteinreihen}, Lecture
  Notes in Mathematics, vol. 1219, Springer-Verlag, Berlin, 1986. \MR{923958}

\end{thebibliography}

\end{document}